\documentclass[a4paper,11pt,leqno]{article}

\usepackage[matrix,arrow,tips,curve]{xy}
\usepackage[english]{babel}
\usepackage{amsmath}
\usepackage{amssymb,amsthm}
\usepackage{enumerate}

\newtheorem{thm}{Theorem}[section]
\newtheorem*{thm*}{Theorem}
\newtheorem{lemma}[thm]{Lemma}
\newtheorem{proposition}[thm]{Proposition}
\newtheorem*{proposition*}{Proposition}
\newtheorem{corollary}[thm]{Corollary}

\theoremstyle{definition}
\newtheorem{definition}[thm]{Definition}
\newtheorem{remark}[thm]{Remark}
\newtheorem{example}[thm]{Example}
\newtheorem{parg}[thm]{}

\newcommand{\ph}{\varphi}
\newcommand{\pr}{\mathbb{P}}
\newcommand{\Q}{\mathbb{Q}}
\newcommand{\R}{\mathbb{R}}
\newcommand{\Z}{\mathbb{Z}}
\newcommand{\C}{\mathbb{C}}
\newcommand{\N}{\mathcal{N}_1}
\newcommand{\M}{\mathcal{M}_X}
\newcommand{\NM}{\operatorname{ME}}
\newcommand{\Nu}{\mathcal{N}^1}

\newcommand{\Sing}{\operatorname{Sing}}
\newcommand{\NE}{\operatorname{NE}}
\newcommand{\Exc}{\operatorname{Exc}}
\newcommand{\Lo}{\operatorname{Locus}}
\newcommand{\codim}{\operatorname{codim}}
\newcommand{\dom}{\operatorname{dom}}
\newcommand{\Nef}{\operatorname{Nef}}
\newcommand{\Eff}{\operatorname{Eff}}
\newcommand{\Mov}{\operatorname{Mov}}
\newcommand{\Pic}{\operatorname{Pic}}
\newcommand{\Cox}{\operatorname{Cox}}
\newcommand{\w}{\widetilde}

\newlength{\Mheight}
\newlength{\cwidth}
\title{On the birational geometry of Fano 4-folds}
\author{Cinzia Casagrande}
\date{September 23, 2011}
    
\begin{document}
\maketitle
{\small\tableofcontents}
\section{Introduction}\label{intro}
After Mori and Mukai's classification of Fano $3$-folds with Picard
number $\rho\geq 2$ in the early 80's, it has become a classical subject to
study Fano manifolds via their contractions\footnote{A
  \emph{contraction} is a morphism with connected fibers onto a normal
projective variety.}, using Mori theory.  Indeed
the
Fano condition makes the situation quite special, because the Cone
and the Contraction Theorems hold for the whole cone of effective curves.

It has been conjectured by Hu and Keel \cite{hukeel}, and recently
proved by Birkar, 
Cascini, Hacon, and 
M{\parbox[b][\Mheight][t]{\cwidth}{c}}Kernan
\cite{BCHM}, that the special behaviour of Fano
manifolds with respect to Mori theory is even stronger: in fact, Fano manifolds
are \emph{Mori dream spaces}.

In particular, 
this implies that the classical point of view can be extended from
regular contractions to \emph{rational contractions}. 
If $X$ is a Mori dream space, a rational contraction of $X$ is a
 rational map $f\colon X\dasharrow Y$ which factors as a
finite sequence of flips, followed by a regular contraction. 
Equivalently, $f$ can be seen as a regular contraction of a
\emph{small $\Q$-factorial modification} of $X$, that is, a variety
related to $X$ by a sequence of flips.

In this paper we use properties of Mori dream spaces to study rational
contractions of a smooth Fano $4$-fold $X$. In particular,
we are interested in bounding the Picard number $\rho_X$
of $X$.

We recall that $\rho_X=b_2(X)$ is a topological invariant of Fano
$4$-folds, whose maximal value is not known.
By taking
products of Del Pezzo surfaces one gets examples with
$\rho\in\{2,\dotsc,18\}$, while all known examples of Fano $4$-folds
which are not products have $\rho\leq 6$.

Our main result is a bound on $\rho_X$ when $X$ has an elementary
rational contraction of fiber type, or more generally, a
\emph{quasi-elementary} rational contraction of fiber type. 
Let us explain  the
terminology: as in the regular case, a rational contraction $f\colon
X\dasharrow Y$ is of fiber type if $\dim Y<\dim X$, and it is
elementary if $\rho_X-\rho_Y=1$.

Quasi-elementary rational
contractions are a special
class of rational contractions of fiber type, 
which includes the elementary ones. 
They share many  useful properties of
the elementary case, for instance the target is again a Mori dream
space. If $f\colon X\to Y$ is a  contraction of fiber type,
then $f$ is quasi-elementary if every curve contracted by $f$ is
numerically equivalent to a one-cycle contained in a general fiber of
$f$. In the case of rational contractions, we give some equivalent
characterizations of being quasi-elementary,
see section~\ref{qe} for more details.

Quasi-elementary (regular) contractions of Fano manifolds have been
studied in  \cite{fanos};
 let us recall what is known in the $4$-dimensional case.
\begin{thm*}[\cite{fanos}, Cor.\ 1.2]
 Let $X$ be a smooth Fano $4$-fold.

\smallskip

\noindent If $X$ has an elementary contraction of fiber type, then $\rho_X\leq
11$, with equality only if $X\cong\pr^1\times\pr^1\times S$ or
$X\cong\mathbb{F}_1\times S$, where $S$ is a surface.

\smallskip

\noindent If $X$ has a  quasi-elementary contraction of fiber type, 
then $\rho_X\leq
18$, with equality only if $X$ is a product of surfaces.
\end{thm*}
\noindent Here is the result in the case of a rational contraction.
\begin{thm}\label{main}
Let $X$ be a smooth Fano $4$-fold. 

\smallskip

\noindent If $X$ has an elementary rational contraction of fiber type, then
$\rho_X\leq 11$. 

\smallskip

\noindent
If $X$ has a quasi-elementary rational contraction of fiber type, 
which is not regular,
then
$\rho_X\leq 17$.
\end{thm}
The strategy for the proof of Th.~\ref{main}
is similar to the one used in \cite{fanos},
via the study of elementary contractions of the target of the rational
contraction of fiber type. We
systematically use properties of Mori dream spaces, and a key
ingredient is a description of non-movable prime divisors in $X$ when
$\rho_X\geq 6$. 
More precisely, we show the following.
\begin{thm}\label{secondo}
Let $X$ be a smooth Fano $4$-fold with $\rho_X\geq 6$, and
 $D\subset X$ a non-movable prime divisor. Then either $D$  is the
locus of an extremal ray of type $(3,2)$,\footnote{See on
  p.~\pageref{oscar} for the terminology.} or
there exists a diagram: 
$$\xymatrix{
X\ar@{-->}[r]&{\widetilde{X}}\ar[d]^{f}\\
& Y}$$
where $X\dasharrow\w{X}$ is a sequence of at least $\rho_X-4$ flips,
$f$ is an elementary divisorial
 contraction with exceptional locus the transform $\w{D}$ of $D$, and
 one of the following
holds:
\begin{enumerate}[$\bullet$]
\item 
$Y$ is smooth and Fano, $f$ is
the blow-up of a smooth curve, and $\w{D}$ is a
$\pr^2$-bundle over a smooth curve;
\item 
$Y$ is smooth and Fano, $f$ is
the blow-up of a point, and $\w{D}\cong\pr^3$;
\item 
$\w{D}$ is isomorphic to a
quadric, $f(\w{D})$ is a factorial and terminal singular point, 
and $Y$ is Fano.
\end{enumerate}
\end{thm}

We finally apply these results to Fano $4$-folds $X$ with $c_X=1$ or
$c_X=2$. 
Let us recall from \cite{codim}
that $c_X$ is an invariant of a Fano manifold $X$, defined as
follows. For any prime divisor $D\subset X$, we consider the
restriction map $H^2(X,\R)\to H^2(D,\R)$, and we set:
$$c_X:=\max\left\{\dim\ker\left(H^2(X,\R)\to H^2(D,\R)\right)
\,|\,D\text{ is a prime divisor in
}X\right\}\in\{0,\dotsc,\rho_X-1\}.$$ 
By \cite[Th.~3.3]{codim} we have $c_X\leq 8$ for any smooth
Fano manifold $X$,
 and if $c_X\geq 4$, then $X$ is a
product of a Del Pezzo surface with another Fano manifold.

In particular, in dimension $4$, we have $\rho_X\leq 18$ as soon as
$c_X\geq 4$. Moreover when $c_X=3$ we know after \cite{codim}
that $\rho_X\leq 8$ (see Th.~\ref{cod}). Therefore in order to study
Fano $4$-folds with large Picard number, we can reduce to the case
$c_X\leq 2$; this is used throughout the paper. 
In the last section we show the following.
\begin{thm}\label{terzo}
Let $X$ be a smooth Fano $4$-fold with
$c_X\in\{1,2\}$. Then either $\rho_X\leq 12$, or $X$ is the blow-up of
another Fano $4$-fold along a smooth surface.
\end{thm}

\medskip

\noindent{\bf Outline of the paper.}
Section~\ref{mds} concerns
Mori dream spaces. In section~\ref{survey} we recall from \cite{hukeel}
the main geometrical properties of Mori dream spaces;
then in section~\ref{qe}
we define quasi-elementary rational contractions and explain some of
their properties.

In section~\ref{leonardo} we move to Fano $4$-folds. We first give in
section~\ref{mattia} some elementary properties of small $\Q$-factorial
modifications and rational contractions of Fano $4$-folds. Then in section~\ref{pn}
we recall some results needed from \cite{codim}, and study the
implications
 on prime divisors in a small $\Q$-factorial
modification of a Fano $4$-fold. Finally in section~\ref{nm} 
we show Th.~\ref{secondo}
on non-movable prime divisors.

In section \ref{rcft}  we show Th.~\ref{main}. We study first  the
case where the target is a surface in section~\ref{surf}, and then the case where
the target has dimension $3$ in section~\ref{3folds} (the case where the target is a curve is easier and is treated in section~\ref{mattia}).

Finally in section \ref{ultima} we show Th.~\ref{terzo}.

\bigskip

\noindent{\bf Acknowledgements.} Part of this paper has been written
during a stay at the  Ludwig-Maximilians University in Munich, in
spring 2010. I am grateful to Professor Andreas Rosenschon and to the
 Mathematisches Institut for the kind hospitality. I also thank the referee for some useful remarks.

\bigskip

\noindent{\bf Notation and terminology}

\smallskip

\noindent We work over the field of complex numbers.\\
A \emph{manifold} is a smooth algebraic variety.\\
A \emph{divisor} is a Weil divisor.\\
If $f\colon X\dasharrow Y$ is a rational map, $\dom(f)$ is the largest
open subset of $X$ where $f$ is regular.

\medskip

\noindent Let $X$ be a normal  projective variety. \\
A \emph{contraction} 
of $X$ is a morphism 
with connected fibers $f\colon X\to Y$
onto a normal projective variety. We will sometimes 
consider the case where $X$ and
$Y$ are quasi-projective and $f$ is a projective morphism; in this case we 
call $f$ a \emph{local contraction}. \\
$\mathcal{N}^1(X)$ (respectively $\mathcal{N}_1(X)$) 
is the $\R$-vector space of Cartier divisors (respectively one-cycles)
with
real coefficients, modulo numerical equivalence.\\
$\Nef(X)\subset\Nu(X)$ is the cone of nef classes.\\
$\Eff(X)\subset\Nu(X)$ is the convex cone generated by classes of
effective divisors, and $\overline{\Eff}(X)$ is its closure.

\medskip

\noindent Let $X$ be a normal and  $\Q$-factorial projective variety.\\
The \emph{anticanonical degree} of a curve $C\subset X$ is $\,-K_X\cdot C$.\\
For any closed subset $Z$ of $X$, $\N(Z,X):=i_*(\N(Z))\subseteq\N(X)$,
where $i\colon Z\hookrightarrow X$ is the inclusion.\\
$[D]$ is the numerical equivalence class in $\mathcal{N}^1(X)$ 
 of a
divisor $D$ in $X$, and similarly $[C]\in\N(X)$ for a curve $C\subset X$.\\ 
$\equiv$ stands for numerical equivalence.\\
For any subset $H\subseteq \N(X)$, $H^{\perp}:=
\{\gamma\in\Nu(X)\,|\,h\cdot\gamma=0\text{ for every }h\in H\}$, and
similarly if $H\subseteq \Nu(X)$.
For any divisor $D$ in $X$,
$D^{\perp}:=[D]^{\perp}\subseteq \N(X)$.\\
A divisor $D$ in $X$ is \emph{movable} if its stable base locus has
codimension at least $2$.
$\Mov(X)\subset\Nu(X)$ is the convex cone generated by classes of
movable divisors.\\
$\NE(X)\subset\N(X)$ is the convex cone generated by classes of
effective curves, and $\overline{\NE}(X)$ is its closure.\\
$\NM(X)\subset\N(X)$ is the cone dual to
$\overline{\Eff}(X)\subset\Nu(X)$. \\
Let $f\colon X\to Y$ be a contraction.
The exceptional locus $\Exc(f)$ is the set of points of $X$
where $f$ is not an isomorphism. 
If $D$ is a divisor in $X$, we say that $f$ is
\emph{$D$-positive} (respectively \emph{$D$-negative}) 
if $D\cdot C>0$ (respectively
$D\cdot C<0$) for every curve $C\subset X$ such that
$f(C)=\{pt\}$. When $D=K_X$, we just say \emph{$K$-positive} (or 
\emph{$K$-negative}).

We consider the push-forward of one-cycles $f_*\colon\N(X)\to\N(Y)$, and 
set $\NE(f):=\overline{\NE}(X)\cap\ker f_*$. We also say that $f$ is the
contraction of $\NE(f)$.

The contraction $f$ is \emph{elementary} if $\rho_X-\rho_Y=1$. In this
case we say that \emph{$f$ is of type $(a,b)$} if $\dim\Exc(f)=a$ and
$\dim f(\Exc(f))=b$.

We will use greek letters $\sigma,\tau,\eta,$ etc.\ to denote convex
polyhedral cones and their faces in $\N(X)$ or $\Nu(X)$.

If $\sigma\subseteq\N(X)$
  is a convex polyhedral cone and $\sigma^{\vee}\subseteq\Nu(X)$ its
  dual cone, there is a   
natural bijection between the faces of $\sigma$ and those of
$\sigma^{\vee}$, given by $\tau\mapsto \tau^{\star}:=\sigma^{\vee}\cap
\tau^{\perp}$ for every face $\tau$ of $\sigma$.

An \emph{extremal ray} of $X$ is a one-dimensional
face of $\overline{\NE}(X)$.\label{oscar}

Consider an elementary contraction $f\colon X\to Y$ and the extremal
ray $\sigma:=\NE(f)$. We say that $\sigma$ is birational, divisorial,
 small, or of type $(a,b)$,
if $f$ is. We set $\Lo(\sigma):=\Exc(f)$, namely $\Lo(\sigma)$ is the
union of the curves whose class belongs to $\sigma$. If $D$ is a
divisor in $X$, we say that
$D\cdot\sigma>0$ if $D\cdot C>0$ for a curve $C$ with $[C]\in\sigma$,
equivalently if $f$ is $D$-positive; similarly for $D\cdot\sigma=0$ and
$D\cdot\sigma<0$.

Suppose that $f\colon X\to Y$ is a small elementary
contraction, and let $D$ be a divisor in $X$ such that $f$ is
$D$-negative. 
The \emph{flip of $f$} is a birational map $g\colon X\dasharrow\w{X}$
which fits into a commutative diagram:
$$\xymatrix{X\ar@{-->}[rr]^g\ar[dr]_f &&{\w{X}}\ar[dl]^{\w{f}}\\
&Y&
}$$
where $\w{X}$ is a normal and $\Q$-factorial projective variety, $g$
is an isomorphism in codimension one, and $\w{f}$ is a
$\w{D}$-positive, small elementary contraction ($\w{D}$ the transform
of $D$ in $\w{X}$). If the flip exists, it
is unique and does not depend on $D$, see \cite[Cor.~6.4  and
  Def.~6.5]{kollarmori}. We also say that $g$ is the flip of the
small extremal ray $\NE(f)$, and that $g$ is a \emph{$D$-negative
flip}. Similarly, if $B$ is a divisor on $X$ such that $f$ is
$B$-positive, we say that $g$ is a $B$-positive flip. Finally, 
 when $D=K_X$, we just say $K$-positive or $K$-negative.

Suppose that $X$ is a projective $4$-fold 
and that $f\colon X\to Y$ is an elementary
contraction. We say that $f$ is of type $(3,2)^{sm}$ if it is
birational
 and every fiber has dimension at most $1$, equivalently if $Y$
is smooth and $f$ is the blow-up of a smooth surface 
(see Th.~\ref{smallfibers}).
\section{Mori dream spaces}\label{mds}
\subsection{A brief survey}\label{survey}
In this section we recall from \cite{hukeel} the definition and 
the main geometrical properties of Mori dream spaces. It is meant as
a quick introduction, and contains no new results; we provide proofs
of some elementary properties for which we could not find an easy
reference.
\begin{definition}
Let $X$ be a normal and $\Q$-factorial projective variety.
A {\bf small $\Q$-factorial modification (SQM)} of $X$ is a normal and
$\Q$-factorial projective variety $\w{X}$, together with a
birational map $f\colon X\dasharrow \w{X}$ which is an isomorphism
in codimension $1$.
\end{definition}
\noindent Flips are examples of SQMs.
\begin{definition}[\cite{hukeel}, Def.~1.10]
Let $X$ be 
a normal and $\Q$-factorial projective variety, with 
finitely generated Picard group. We say that $X$ is a {\bf Mori dream
  space} if there are a finite number of SQMs $f_j\colon X\dasharrow
X_j$ such that:
\begin{enumerate}[$(i)$]
\item for every $j$, $\Nef(X_j)$ is a polyhedral cone,
generated by the classes of finitely many semiample divisors;
\item $\Mov(X)=\bigcup_j f_j^*(\Nef(X_j))$.
\end{enumerate}
\end{definition}
Notice that if $X$ is a normal and $\Q$-factorial projective
variety having a SQM $\w{X}$ which is a Mori dream space, then $X$
itself is a Mori dream space.

Let $X$ be a Mori dream space.
We consider the following cones in $\Nu(X)$:
$$\Nef(X)\subseteq\Mov(X)\subseteq\Eff(X).$$
 All three 
are closed and 
polyhedral (see \cite[Prop.~1.11(2)]{hukeel}), 
and have dimension\footnote{The dimension of a cone in
  $\R^m$ is the dimension of its linear span.} $\rho_X$.
By condition $(ii)$, one of the SQMs $f_j$ must be the
identity of $X$, and by $(i)$
 $\Nef(X)$ is generated by the classes of finitely many semiample
divisors.  
In particular
this implies that the association 
$$\left(f\colon X\to Y\right)\ \longmapsto\ f^*\left(\Nef(Y)\right)$$
yields a bijection between the set of
contractions of $X$ and the set of faces of $\Nef(X)$.
\begin{definition}\label{ratcontr}
Let $X$ be a Mori dream space. A {\bf rational contraction} of $X$ is
a rational map $f\colon X\dasharrow Y$ which factors as $X\dasharrow\w{X}\to
Y$, where $X\dasharrow\w{X}$ is a SQM, and $\w{X}\to Y$ a (regular)
contraction.
\end{definition}
\noindent (In \cite{hukeel} the terminology ``contracting 
rational map'' is also used.) Let us notice that the definition 
\cite[Def.~1.1]{hukeel} is more general, because $X$ is just assumed
to be a
normal projective variety; when $X$ is a Mori dream space, the two
notions coincide, by \cite[Prop.~1.11]{hukeel}.

Every SQM of $X$ factors as a finite sequence of flips (see
\cite[Prop.~1.11]{hukeel}), therefore a
rational contraction can equivalently be described as a rational map
which factors as a finite sequence of flips followed by a contraction.
\begin{remark}\label{contracting}
Let $X$ be a Mori dream space, $Y$ a normal projective variety, and
$f\colon X\dasharrow Y$ a dominant rational map with connected
fibers.\footnote{Namely, a resolution of $f$ has connected
fibers; this does not depend on the resolution, see
\cite[Def.~1.0]{hukeel}.}
If there exist open subsets $U\subseteq X$
and $V\subseteq Y$ such that $\codim(Y\smallsetminus V)\geq 2$ and
$f_U\colon U\to V$ is a regular contraction, then $f$ is a rational
contraction. When $f$ is birational, also the converse holds.

Indeed consider a resolution of $f$:
$$\xymatrix{&{\widehat{X}}\ar[dl]_g\ar[dr]^{\widehat{f}}&\\
X\ar@{-->}[rr]^f&&Y}$$
where $\widehat{X}$ is normal and projective, and $g$ is birational
and an
isomorphism over $\dom(f)$. Then $Y\smallsetminus V\supseteq
\widehat{f}(\Exc(g))$, so that $\codim \widehat{f}(\Exc(g))\geq
2$. Hence if $D$ is an effective, $g$-exceptional Cartier divisor in
$\widehat{X}$, then
$(\widehat{f})_*\mathcal{O}_{\widehat{X}}(D)=\mathcal{O}_Y$
(\emph{i.e.}\ $D$ is $\widehat{f}$-fixed, in the terminology of
\cite{hukeel}). Thus $f$ is a rational contraction by
\cite[Def.~1.1 and Prop.~1.11]{hukeel}.
\end{remark}
If  $f\colon X\dasharrow Y$ is a rational contraction, there is a
well-defined injective linear map $f^*\colon \Nu(Y)\to\Nu(X)$, such
that  $f^*(\Nef(Y))\subseteq\Mov(X)$.
The bijection between the contractions of $X$ and the faces of
$\Nef(X)$ generalizes to rational contractions in the following way.
Define
$$\M:=\left\{f^*(\Nef(Y))\,|\,f\colon X\dasharrow Y\text{ is a
  rational contraction of $X$}\right\}.$$
Then we have the following.
\begin{proposition}[\cite{hukeel}, Prop.\ 1.1(3)]
The set  $\M$ is a fan\footnote{We recall that a \emph{fan} $\Sigma$ in
  $\R^m$ is a finite set of 
convex polyhedral cones in 
$\R^m$, with the following properties:
1)
for every $\sigma\in\Sigma$, every face of $\sigma$ is in $\Sigma$;
2)
for every $\sigma,\tau\in\Sigma$, $\sigma\cap\tau$ is a face of both
$\sigma$ and $\tau$.} 
in $\Nu(X)$. The union of the cones
in $\M$ is $\Mov(X)$, and every face of $\Mov(X)$ is a union of
cones in $\M$. Moreover,
 the association $$\left(f\colon X\dasharrow Y\right)\ 
\longmapsto\ f^*\left(\Nef(Y)\right)$$ 
gives a bijection between the set of rational
contractions of $X$ and $\M$.
\end{proposition}
Here are some properties of this bijection:
\begin{enumerate}[$\bullet$]
\item
if $\sigma\in\M$ and $f\colon X\dasharrow Y$ is the corresponding
contraction, then $\dim\sigma=\rho_Y$;
\item $f$ is regular if and only if $\sigma\subseteq\Nef(X)$; in
  particular $\Nef(X)\in\M$ corresponds to the identity of $X$;
\item 
$f$ is \emph{of fiber type} (\emph{i.e.}\ $\dim Y<\dim X$) 
if and only if $\sigma$ is contained in the
  boundary of $\Eff(X)$;  
\item $f$ is a SQM if and only if $\dim\sigma=\rho_X$;
\item given two cones $\sigma_1,\sigma_2\in\M$ with corresponding
  rational contractions $f_i\colon X\dasharrow Y_i$, then
  $\sigma_1\subseteq
\sigma_2$ if and only if there is a regular contraction $g\colon
  Y_2\to Y_1$ such that the following diagram commutes:
$$\xymatrix{
&X\ar@{-->}[dl]_{f_1}\ar@{-->}[dr]^{f_2}&\\
{Y_1}&&{Y_2}\ar[ll]_g}$$
In particular, given $f_1\colon X\dasharrow Y_1$, the factorizations
$X\dasharrow\w{X}\to Y_1$ of $f_1$ with $X\dasharrow\w{X}$ a SQM 
correspond to
$\rho_X$-dimensional cones in $\M$ containing $\sigma_1$. 
\end{enumerate}
\begin{example}[Elementary rational contractions]\label{example}
Let $f\colon X\dasharrow Y$ be a rational contraction. We say that $f$
is \emph{elementary} if $\rho_X-\rho_Y=1$, equivalently if $\dim\sigma=
\rho_X-1$, where $\sigma\in\M$ is the cone corresponding to $f$.
As in the regular case,
we have three
possibilities:
\begin{enumerate}[$(i)$]
\item
if $\sigma$ is in the interior of $\Mov(X)$, then $f$ is an elementary
small contraction of a SQM of $X$;
\item
if $\sigma$ lies on the boundary of $\Mov(X)$ but in the interior of
$\Eff(X)$, then $f$ is an elementary divisorial contraction 
of a SQM of $X$;
\item
if $\sigma$ lies on the boundary of $\Eff(X)$, then $f$ is an elementary fiber
type contraction of a SQM of $X$.
\end{enumerate}
As in the regular case, we will say that $f$ is \emph{small} in case
$(i)$, \emph{divisorial} in case $(ii)$.
\end{example}
\begin{example}[Flips] Let  $f\colon X\to Y$ be
  a small elementary contraction, and consider
 $\sigma:=f^*(\Nef(Y))\in\M$. 
  The cone $\sigma$ is a facet of $\Nef(X)$ and lies in the interior
  of $\Mov(X)$, therefore there exists a unique $\rho_X$-dimensional
  cone $\tau\in\M$ such that $\sigma=\Nef(X)\cap\tau$. Let $g\colon
  X\dasharrow\w{X}$ be the SQM corresponding to $\tau$; then $g$ is the
  flip of $f$.
\end{example}
\begin{remark}\label{targetMDS}
Let $X$ be a Mori dream space and $f\colon X\dasharrow Y$ a rational
contraction. Suppose that
$Y$ is $\Q$-factorial. Then $Y$ is a Mori dream space, and for every
rational contraction $g\colon Y\dasharrow Z$, the composition
$g\circ f\colon
X\dasharrow Z$ is again a rational 
contraction.  
\end{remark}
\begin{proof}
The statement is clear from the definitions if $f$ is a SQM. In
general, we factor $f$ as 
 $X\dasharrow\w{X}\stackrel{\w{f}}{\to} 
Y$, where $\w{X}$ is a
SQM of $X$, and $\w{f}$ is a regular contraction. 
Since $\w{X}$
is a Mori dream space, and $g\circ f\colon X\dasharrow Z$ 
is a rational contraction if and
only if  $g\circ \w{f}\colon\w{X}\dasharrow Z$ is,
we can assume that $f$ is regular.

Now $f^*\colon\Pic(Y)\to\Pic(X)$ is injective, hence $Y$ has
finitely generated Picard group.
Then we can define the Cox rings $\Cox(Y)$ and $\Cox(X)$ of $Y$ and
$X$, see \cite[Def.~2.6]{hukeel}. By \cite[Prop.~2.9]{hukeel}
 $Y$ is a Mori
dream space if and only if 
$\Cox(Y)$ is a finitely generated $\C$-algebra, and for the same
reason $\Cox(X)$ is a finitely generated $\C$-algebra.

 We have $f^*(\Eff(Y))=\Eff(X)\cap
f^*(\Nu(Y))$, so that $f^*(\Eff(Y))$ is closed and is
a convex polyhedral cone. Moreover,
 via $f^*$, we can see $\Cox(Y)$
as a subalgebra of $\Cox(X)$, graded by the subsemigroup of integral
points of $f^*(\Eff(Y))$. This kind of subalgebra is called a Veronese
subalgebra; since $\Cox(X)$ is finitely generated, the
same holds for $\Cox(Y)$, see \cite[Prop.~1.2.2]{coxbook}. Thus $Y$ is
a Mori dream space.

Let us show that $g\circ f$ is a rational contraction.
We factor $g$ as $Y\stackrel{h}{\dasharrow}\w{Y}\stackrel{\w{g}}{\to} Z$,
where $h$ is a SQM and $\w{g}$ a regular contraction, and first
consider  $h\circ f\colon X\dasharrow \w{Y}$. We have 
$\codim(\w{Y}\smallsetminus\dom(h^{-1}))\geq 2$, and $(h\circ
f)_{f^{-1}(\dom(h))}\colon f^{-1}(\dom (h))\to \dom(h^{-1})$ is a
regular contraction, so $h\circ f$ is a rational contraction by
Rem.~\ref{contracting}. 

It is clear from Def.~\ref{ratcontr} that the composition of a
rational contraction with a regular contraction is again a rational
contraction; since $g\circ f=\w{g}\circ(h\circ f)$, we are done.
\end{proof}
\begin{remark}\label{aggiunta}
If $X$ is a Mori dream space and $f\colon X\to Y$ is a contraction, then $(\ker f_*)^{\perp}=f^*(\Nu(Y))$. In other words, for any divisor $D$ in $X$, one has $D^{\perp}\supseteq\ker f_*$ if and only if $[D]\in f^*(\Nu(Y))$.
Indeed it is easy to see that $(\ker f_*)^{\perp}\supseteq f^*(\Nu(Y))$, and since both subspaces have dimension $\rho_Y$, they must coincide.
\end{remark}
\begin{parg}{\bf Mori programs.} Let $X$ be a Mori dream space, and
$D$ a divisor in $X$. A Mori program
for $D$ is a sequence of varieties and birational maps
\stepcounter{thm}\begin{equation}\label{MMP}
X=X_0\stackrel{f_0}{\dasharrow}X_1{\dasharrow}\cdots\dasharrow
X_{k-1}\stackrel{f_{k-1}}{\dasharrow}X_k\end{equation}
\newcounter{MMPuno}\newcounter{MMPdue}\newcounter{MMPtre}
such that:
\stepcounter{thm}
\begin{enumerate}[(\thethm)]
\item\setcounter{MMPuno}{\value{thm}}
\stepcounter{thm}
every $X_i$ is a normal and $\Q$-factorial projective variety;
\item\setcounter{MMPdue}{\value{thm}}
\stepcounter{thm}
for every $i=0,\dotsc,k-1$ there is a birational, $D_i$-negative
 extremal ray $\sigma_i$ of
$\NE(X_i)$, such that
 $f_i$ is either the contraction of $\sigma_i$ (if divisorial), or its
flip (if small). The divisor $D_{i+1}$ is defined as
$(f_i)_*(D_i)$ if $f_i$ is a divisorial contraction, 
as the transform of $D_i$ if $f_i$ is a flip;
\item\setcounter{MMPtre}{\value{thm}}
either $D_k$ is semiample, or there exists a $D_k$-negative
 elementary contraction
of fiber type $f_k\colon X_k\to Y$.
\end{enumerate}
An important property of Mori dream spaces is that  one can run a Mori
program for any divisor 
$D$, see \cite[Prop.~1.1(1)]{hukeel}; moreover, the choice of the
extremal rays $\sigma_i$ is arbitrary among those having negative
intersection with $D_i$.
\begin{remark}\label{output}
A Mori program for $D$ ends with a fiber type
contraction if and only if $[D]\not\in\Eff(X)$.\end{remark}
\end{parg}
\begin{parg}{\bf Cones of curves.}
In $\N(X)$ we have dual cones:
$$\NM(X):=\Eff(X)^{\vee}\subseteq\Nef(X)^{\vee}=\NE(X).$$
Recall that by \cite{BDPP}, for any projective variety $X$, the dual
$\NM(X)$
of the cone $\overline{\Eff}(X)$ is the closure of the convex cone
generated by classes of irreducible curves belonging to a covering
family of curves.

When $X$ is a Mori dream space, the cone $\NM(X)$ is polyhedral,
because $\Eff(X)$ is. Using the same techniques as in \cite{carolina}
(in a much simpler situation), one
can see that 
 every one-dimensional face of $\NM(X)$ contains the class of a
curve moving in a covering family. The proof of the following Lemma is
adapted from \cite[Lemma~5.1 and Th.~5.2]{carolina}; we 
 write it explicitly
for the reader's convenience.
\begin{lemma}
Let $X$ be a Mori dream space and $\sigma$ a one-dimensional face of
$\NM(X)$. Then there exists a Mori program on $X$ ending with a fiber
type contraction:
$$X\dasharrow X'\stackrel{f}{\longrightarrow}Y$$
such that if $C\subset X$ is the transform of a general curve in a
general fiber of $f$, then $[C]\in\sigma$. 
\end{lemma}
\begin{proof}
Let $B$ be an effective divisor such that
$B^{\perp}\cap\NM(X)=\sigma$, let $H$ be an ample divisor, and set
$D:=B-H$. Since $[B]$ lies on the boundary of $\Eff(X)$ and $[H]$ in
its interior, we have $[D]\not\in\Eff(X)$. By Rem.~\ref{output}, every
Mori program for $D$ ends with a fiber type contraction. 

We run a Mori program for $D$ with scaling of $H$, see 
\cite[\S~3.10]{BCHM} and \cite[\S~3.8]{carolina}.
Concretely,
this means a sequence as \eqref{MMP}, where at each step the extremal
ray $\sigma_i$ is chosen in a prescribed way. 
At the first step, we choose a facet of $\Nef(X)$
 met by moving from $[D]$ to $[H]$ along the segment $s$ joining them
in $\Nu(X)$. This facet corresponds  to a $D$-negative
 extremal ray of $\NE(X)$; this will be $\sigma_0$. 
This process can be repeated at
each step, using $H_i$ in $X_i$, where $H_i:=(f_{i-1}\circ\cdots\circ
f_0)_*(H)$.

The segment $s$ meets the boundary of $\Eff(X)$ at the point
$[B]/2=([D]+[H])/2$. The key remark, made in
\cite[Lemma~5.1]{carolina}, is that 
for every $i\in\{1,\dotsc,k\}$
 the segment from $[D_i]$ to $[H_i]$ in $\Nu(X_i)$ meets the
boundary of $\Eff(X_i)$ at the point
$([D_i]+[H_i])/2$. Indeed, suppose that this is true for $X_{i-1}$,
and consider $f_{i-1}\colon X_{i-1}\dasharrow X_i$. The statement is
clear if $f_{i-1}$ is a flip, thus let's assume that it is a
divisorial contraction. 

We know that $(1-t)[D_{i-1}]+t[H_{i-1}]\in\Eff(X_{i-1})$ for $t\in [1/2,1]$,
and  $(1-t)[D_{i-1}]+t[H_{i-1}]\not\in\Eff(X_{i-1})$
for $t\in [0,1/2)$. Moreover
  $(1-t)D_{i}+tH_{i}=(f_{i-1})_* ((1-t)D_{i-1}+tH_{i-1})$, so that
  again $(1-t)[D_{i}]+t[H_{i}]\in\Eff(X_{i})$ if $t\in [1/2,1]$. 

We have
  $D_{i-1}\cdot\NE(f_{i-1})<0$; moreover, by the choice of $\NE(f_{i-1})$,
 there exists $t_0\in [1/2,1]$ such
  that $((1-t_0)D_{i-1}+t_0H_{i-1})\cdot\NE(f_{i-1})=0$. Hence
$$\left((1-t)D_{i-1}+tH_{i-1}\right)\cdot\NE(f_{i-1})<0\quad\text{for
    every }t<\frac{1}{2}.$$
Therefore if
  $(1-t)[D_{i}]+t[H_{i}]\not\in\Eff(X_{i})$  for some $t\in [0,1/2)$, we can proceed
as in the proof of Rem.~\ref{output} and get a contradiction.

\medskip

In the end we get an elementary contraction of fiber type $f_k\colon
X_k\to Y$ such that $((1-t_k)D_{k}+t_kH_{k})\cdot\NE(f_k)=0$ for some
$t_k\in(0,1]$. Then $(1-t_k)[D_{k}]+t_k[H_{k}]$ lies on the boundary of
$\Eff(X_k)$, and by what we proved above, $t_k=1/2$. This means that
if $C\subset X$ is the transform of a general curve in a
general fiber of $f_k$, then $B\cdot C=0$, therefore
$[C]\in\sigma$. 
\end{proof}
\end{parg}
\begin{parg}{\bf Non-movable prime divisors.}
We conclude this section by showing that non-movable prime divisors in $X$ 
are exactly the divisors which become exceptional on some SQM of $X$.
Notice that if $D$ is a divisor in $X$, then $D$ is movable
(\emph{i.e.}\ the stable locus of $D$ has codimension at least $2$) if
and only if $[D]\in\Mov(X)$.
\begin{remark}\label{nonmovable}
Let $X$ be a Mori dream space, and $D\subset X$ a prime divisor. The
following conditions are equivalent:
\begin{enumerate}[$(i)$]
\item $D$ is not movable;
\item there exists a SQM ${X}\dasharrow\w{X}$ such that the transform
  $\w{D}\subset\w{X}$ of $D$ is the exceptional
divisor of an elementary divisorial contraction $\w{X}\to Y$.
\end{enumerate}
Moreover, the association $D\mapsto \R_{\geq 0}[D]$ gives a bijection between:
\begin{enumerate}[$\bullet$]
\item the set of non-movable prime divisors in $X$, and
\item the set of one-dimensional faces of $\Eff(X)$ not contained in
  $\Mov(X)$. 
\end{enumerate}
\end{remark}
Let us point out that after the proof,
 $X\dasharrow\w{X}\to Y$ (notation as in $(ii)$)
is a Mori program for $D$ (ending with zero), so that 
$X\dasharrow\w{X}$
factors as a sequence of
$D$-negative flips. In fact, every Mori program for $D$ takes this form. 
\begin{proof}
Suppose that $D$ is not movable, and 
consider a Mori program for $D$. Since $D$ is effective, by
Rem.~\ref{output}
the program
must end with $D$ becoming nef. On the other hand, there is no SQM of
$X$ where $D$ is nef, because $D$ is not movable. Therefore in the Mori
program  some divisorial contraction must occur. Let $f\colon \w{X}\to
Y$ be the first divisorial contraction: then the previous steps are
flips, hence $X\dasharrow\w{X}$ is a SQM (possibly
$\w{X}=X$). Moreover since $\w{D}\cdot\NE(f)<0$ and $\w{D}$ is a prime
divisor, we have $\w{D}=\Exc(f)$. Since $f_*(\w{D})=0$, the divisorial contraction  $f\colon \w{X}\to
Y$ is the last step of the Mori program.

Conversely, if $(ii)$ holds, then $\w{D}$ is not movable, hence neither
is $D$. 

Finally, suppose that $(i)$ and $(ii)$ hold, 
and let $D_1,D_2\subset\w{X}$
be prime divisors such that
$a_1D_1+a_2D_2\equiv\w{D}$, $a_i\in\R_{>0}$. Since
$\w{D}\cdot\NE(f)<0$, there exists $i\in\{1,2\}$ such that $D_i
\cdot\NE(f)<0$, hence $D_i=\w{D}$. This implies that $D_1=D_2=\w{D}$,
therefore $\R_{\geq 0}[\w{D}]$ is a one-dimensional face of
$\Eff(\w{X})$. Similarly, one shows that $\w{D}$ is the unique prime
divisor whose class belongs to  this face.
\end{proof}
We will also need the following.
\begin{remark}\label{factor}
Let $X$ be a Mori dream space, $g\colon X\to Z$ a contraction, and 
$D\subset X$ a non-movable 
prime divisor such that $g(D)=\{pt\}$. 
Then there exists a commutative diagram:
$$\xymatrix{{{X}}\ar[d]_{g}\ar@{-->}[r]& {\w{X}}\ar[d]^{f}\\
{Z}&{Y}\ar[l]_{h}
}$$
where $X\dasharrow\w{X}$ is a SQM which factors as a sequence of
$D$-negative  flips, $f$  is an elementary
divisorial contraction        
with exceptional divisor (the transform of) $D$, and $h$
is a contraction.
\end{remark}
\begin{proof}
By Rem.~\ref{nonmovable}, there are a birational map $X\dasharrow\w{X}$ which
factors as a sequence of $D$-negative flips, and an elementary
divisorial contraction $f\colon \w{X}\to Y$ with exceptional divisor
the transform of $D$.
If $\sigma$ is a $D$-negative extremal ray of $\NE(X)$, then
$\Lo(\sigma)\subseteq D$, so that $g(\Lo(\sigma))=\{pt\}$ and
$\sigma\subseteq
\NE(g)$. Iterating this reasoning, we see that the rational map
$h\colon Y\dasharrow Z$ is indeed regular.
\end{proof}
\end{parg}
\subsection{Quasi-elementary rational contractions}\label{qe}
In this section we introduce a special class of 
rational contractions of
fiber type of Mori dream spaces, called 
quasi-elementary contractions,
 which share many good properties
of elementary rational contractions of fiber type. The notion of
quasi-elementary contraction was first introduced in \cite{fanos}, but
in a different context: there the objects were regular, $K$-negative
contractions of a smooth projective variety. Here, since we are
considering Mori dream spaces, we do not need to assume
$K$-negativity. 

Let $X$ be a Mori dream space and $f\colon X\to Y$ a contraction.
Recall that $$\NE(f):=\ker f_*\cap\NE(X)$$ 
is a face of $\NE(X)$, corresponding
to the face 
$f^*(\Nef(Y))$ of $\Nef(X)$. 
In the same way we can associate to $f$ a face
of $\NM(X)$, setting 
$$\NM(f):=\ker f_*\cap\NM(X)=\NE(f)\cap\NM(X).$$
Notice that $\NM(f)$ is non-zero if and only if $f$ is of fiber type.
\begin{lemma}\label{MFS} In the notation above, let $F\subset X$ be a
  general  fiber of $f$, and $i\colon F\hookrightarrow X$ the inclusion.
Then 
$$\NM(f)=i_*(\NM(F)),$$ 
the linear span of $\NM(f)$ is $\N(F,X)$, and  $\dim\NM(f)=\dim\N(F,X)$.
Moreover $\NM(f)^{\star}=\Eff(X)\cap\N(F,X)^{\perp}$ 
is the smallest face of
$\Eff(X)$ containing $f^*(\Eff(Y))$.
\end{lemma}
\begin{proof}
We clearly have $i_*(\NM(F))\subset\ker f_*$. 
Let $D_1,\dotsc,D_r\subset X$ be prime divisors whose classes
generate $\Eff(X)$. Then for every $j\in\{1,\dotsc,r\}$ $D_j$
does not contain $F$, and if
$\gamma\in\NM(F)$ we have
$$i_*(\gamma)\cdot D_j=\gamma\cdot (D_j)_{|F}\geq 0,$$
so that $i_*(\gamma)\in\Eff(X)^{\vee}=\NM(X)$. This shows that
$i_*(\NM(F))\subseteq\ker f_*\cap\NM(X)=\NM(f)$. 

Conversely, let $\sigma$ be a one-dimensional face of $\NM(f)$. By
Rem.~\ref{MFS}, there is a covering family of 
curves $\{C_t\}$ in $X$ whose numerical class belongs to $\sigma$. On
the other hand, since $\sigma\subset\ker f_*$, all these curves are 
contracted to a point by $f$. This means that a subfamily $\{C_{t'}\}$
gives a covering family of curves in $F$, hence $[C_{t'}]\in\NM(F)$
and $\sigma\subseteq i_*(\NM(F))$. Therefore $\NM(f)=i_*(\NM(F))$.

Now since $\NM(F)$ generates $\N(F)$, we get that $\N(F,X)=i_*(\N(F))$
is the linear span of $\NM(f)$ in $\N(X)$, and
$\dim\NM(f)=\dim\N(F,X)$. 

\medskip

For the last statement, let $\tau$ be a face of $\NM(X)$ and
$\tau^{\star}$ the corresponding face of $\Eff(X)$. By
the  definition of $\tau^{\star}$, if $H\subseteq\N(X)$ is a linear
subspace, then 
$\tau\subset H$ if and only if $\tau^{\star}\supseteq\Eff(X)\cap H^{\perp}$.
Now take $H=\ker f_*$. Since $H^{\perp}=f^*(\Nu(Y))$ (see Rem.~\ref{aggiunta}), 
we get:
$$\tau\subseteq\NM(f)\ \Longleftrightarrow \ \tau^{\star}\supseteq \Eff(X)\cap 
f^*(\Nu(Y))=f^*(\Eff(Y)).$$
\end{proof}
\begin{proposition}\label{sabrina}
Let $X$ be a Mori dream space,
 $f\colon X\dasharrow Y$ a rational
contraction, and $\sigma\in\M$ the corresponding
cone.  Let $X\dasharrow\w{X}\stackrel{\w{f}}{\to}Y$ be a factorization
of $f$ as a SQM followed by a contraction, 
and  let $F\subset\w{X}$ be a general fiber of
$\w{f}$. 

The following properties are equivalent:
\begin{enumerate}[$(i)$]
\item $\N(F,\w{X})=\ker\w{f}_*$;
\item  $\dim\N(F,\w{X})=\rho_X-\rho_Y$;
\item $\dim\NM(\w{f})=\rho_X-\rho_Y$;
\item $\sigma$ is
  contained in a face of $\Eff(X)$ of the same dimension as $\sigma$
  (that is, $\rho_Y$);
\item $f^*(\Eff(Y))$ is a face of $\Eff(X)$.
\end{enumerate}
\end{proposition}
\begin{definition}
We say that $f$ is \emph{quasi-elementary} if the equivalent
conditions above are satisfied and $f$ is non-trivial
(\emph{i.e.}\ $f$ is not an isomorphism nor constant). In particular,
$f$ must be of fiber type. 
\end{definition}
\noindent Notice that $f$ is quasi-elementary if and only if $\w{f}$
is (notation as in Prop.~\ref{sabrina}).
\begin{proof}[Proof of Prop.~\ref{sabrina}]
Up to replacing $X$ by $\w{X}$, we can assume that $f\colon X\to Y$ is regular.

\medskip

\noindent $(i)\Rightarrow(iii)$ This 
 follows from  Lemma~\ref{MFS}.

\medskip

\noindent$(iii)\Rightarrow(v)$ 
Since $\dim\NM(f)=\dim\ker f_*$, $\ker f_*$ is the linear span of
$\NM(f)$. Therefore 
$$\NM(f)^{\star}=\Eff({X})\cap (\ker f_*)^{\perp}=\Eff({X})\cap
{f}^*(\Nu(Y))= {f}^*(\Eff(Y)).$$

\medskip

\noindent$(v)\Rightarrow(iv)$ This is because
$\sigma=f^*(\Nef(Y))\subseteq{f}^*(\Eff(Y))$.  

\medskip

\noindent$(iv)\Rightarrow(ii)$ Let $\eta$ be the face of $\Eff({X})$
containing $\sigma$ and such that $\dim \eta=\dim\sigma=\rho_Y$. 
Then the linear span of $\eta$ is the same as that of $\sigma$, namely
$f^*(\Nu(Y))$. This gives 
$$\eta^{\star}=\NM(X)\cap(f^*(\Nu(Y)))^{\perp}=\NM(X)\cap\ker f_*=\NM(f),$$
and 
by Lemma~\ref{MFS} we get
$\dim\N(F,X)=\dim\eta^{\star}=\rho_X-\rho_Y$.

\medskip

\noindent$(ii)\Rightarrow(i)$ This follows from $\N(F,X)\subseteq\ker f_*$ and $\dim\ker f_*=\rho_X-\rho_Y$.
 \end{proof}
\begin{remark}\label{pippo}
Let $X$ be a Mori dream space and $f\colon X\dasharrow Y$ a rational
contraction of fiber type with $\dim Y>0$.
\begin{enumerate}[$\bullet$]
\item
If $f$ is elementary, then it is also quasi-elementary.
\item
If $\dim Y=\dim X-1$, then $f$ is elementary if and only if it is
quasi-elementary. 
\item
If $f$ is quasi-elementary and regular, and $F\subset X$ is a general
fiber, then $\rho_X-\rho_Y\leq\rho_F$. 
\end{enumerate}
\end{remark}
If $X$ is a Mori dream space, then $X$ has a (non-trivial) rational
contraction of fiber type if and only if the boundaries of $\Mov(X)$
and $\Eff(X)$ meet outside zero. For the quasi-elementary case we have
the following criterion. 
\begin{corollary}\label{tobia}
Let $X$ be a Mori dream space and $r\in\{1,\dotsc,\rho_X-1\}$. 
Then $X$ has a quasi-elementary
rational contraction $f\colon X\dasharrow Y$ with $\rho_Y=r$
if and only if there exists  an $r$-dimensional
face of $\Mov(X)$ contained in an $r$-dimensional face of $\Eff(X)$.
\end{corollary} 
\begin{proof}
Let $f\colon X\dasharrow Y$ be a quasi-elementary rational
contraction with $\rho_Y=r$, and let $\sigma\in\M$ be the
corresponding cone. Then $\dim\sigma=r$, and by
Prop.~\ref{sabrina} $(iv)$, $\sigma$ is contained in a face $\tau$ of
$\Eff(X)$ with $\dim \tau=r$. There exists a
 face $\eta$ of $\Mov(X)$ with
$\sigma\subseteq\eta\subseteq\tau$, and we get $\dim\eta=r$.

Conversely, let $\eta$ be a face of $\Mov(X)$ contained in a face $\tau$
of $\Eff(X)$ with $\dim\eta=\dim\tau=r$. Since $\eta$ is a union of
cones in $\M$, we can choose 
 $\sigma\in\M$ such that $\sigma\subseteq\eta$ and $\dim
\sigma=r$. Then the rational contraction corresponding to
$\sigma$ is quasi-elementary again by Prop.~\ref{sabrina} $(iv)$,
and the target has Picard number $r$. 
\end{proof}
\begin{remark}\label{composition}
Let $X$ be a Mori dream space and $f\colon X\dasharrow Y$ a
quasi-elementary rational contraction. Then $Y$ is a Mori dream space,
and if $g\colon Y\dasharrow Z$ is a
quasi-elementary rational contraction, then $g\circ f\colon
X\dasharrow Z$ is again quasi-elementary.
\end{remark}
\begin{proof}
We first show that $Y$ is $\Q$-factorial. 
Up to replacing $X$ with a SQM, we can assume that $f$ is regular.
Let $D\subset Y$ be a prime
divisor in $Y$, and let $D'\subset X$ be a prime divisor such that $f(D')\subseteq D$. 

If $F\subset X$ is a general fiber of $f$, then $F\cap D'=\emptyset$, so that for every curve $C\subseteq F$ we have $D'\cdot C=0$. This gives $(D')^{\perp}\supseteq \N(F,X)$; on the other hand $\N(F,X)=\ker f_*$ because $f$ is quasi-elementary, so that $[D']\in (\ker f_*)^{\perp}=f^*(\Nu(Y))$ 
(see Rem.~\ref{aggiunta}). Hence there exist a Cartier divisor $B$ in $Y$ and an integer $m\in\mathbb{N}$ such that $mD'=f^*(B)$. This shows that $B$ is effective and has support contained in $f(D')\subseteq D$, 
therefore $B=rD$ for some $r\in\mathbb{N}$. Thus $D$ is $\Q$-Cartier, and $Y$ is $\Q$-factorial.

Now applying  Rem.~\ref{targetMDS}
we see that
$Y$ is a Mori dream space and  $g\circ f\colon
X\dasharrow Z$ is a rational contraction. 
Since both $f$ and $g$ are quasi-elementary,
Prop.~\ref{sabrina} $(v)$ says that
$f^*(\Eff(Y))$ is a face of $\Eff(X)$, and $g^*(\Eff(Z))$ is a face of
$\Eff(Y)$. Then $(g\circ
f)^*(\Eff(Z))$ is a face of $\Eff(X)$, thus $g\circ f$ is
quasi-elementary again by  
 Prop.~\ref{sabrina} $(v)$.
\end{proof}
\section{Non-movable prime divisors in a Fano $4$-fold}\label{leonardo}
\subsection{Fano $4$-folds as Mori dream spaces}\label{mattia}
Let us recollect some well-known results which will be used
in the sequel. We recall that by a $4$-fold we always mean a \emph{smooth}
$4$-dimensional algebraic variety. 
\begin{thm}[see \cite{AWaview}, Th.~4.1 and
    references therein]\label{smallfibers}
Let $X$ be a quasi-projective $4$-fold and $f\colon X\to Y$ a local
contraction such that $-K_X$ is $f$-ample. Assume that every fiber of
$f$ has dimension at most $1$. Then $Y$ is smooth and $f$ is either
the blow-up of a smooth surface in $Y$, or a conic bundle.
\end{thm}
\begin{thm}\label{components}
Let $X$ be a projective variety with canonical singularities and $K_X$
Cartier, $f\colon X\to Y$ a $K$-negative contraction with $\dim
X-\dim Y\leq 1$, and $F\subset X$ an isolated $2$-dimensional fiber of
$f$.

Let $T$ be a $2$-dimensional irreducible component of $F_{red}$. Then the
possibilities for $(T,-K_{X|T})$ are the following:
\begin{enumerate}[$(i)$]
\item $(\pr^2,\mathcal{O}_{\pr^2}(e))$ with
  $e=1,2$;
 \item $(S_r,\mathcal{O}_{S_r}(1))$ with $r\geq 2$;
\item $(\mathbb{F}_r,C_0+mB)$ with $r\geq 0$,
  $m\geq r+1$.
\end{enumerate}
Here $S_r$ is the cone over a rational normal curve of degree $r$,
$B\subset\mathbb{F}_r$ is a fiber of the $\pr^1$-bundle, and 
$C_0\subset\mathbb{F}_r$ is a section of the $\pr^1$-bundle with
$C_0^2=-r$.

If moreover $X$ is smooth, then every irreducible component of $F$ has
dimension $2$.
\end{thm}
\begin{proof} If $f$ is birational, this is
 \cite[Th.~1.19 and Prop.~4.3.1]{AWaview}. When $f$ is of fiber type, 
  by \cite[Th.~2.6]{mella} there exists a non-empty open neighbourhood $Y_0$
  of $f(F)$ such that  $-K_X$ is $f$-spanned on $f^{-1}(Y_0)$. Then 
\cite[Th.~1.19 and Prop.~4.3.1]{AWaview} still apply.
\end{proof}
Let $X$ be a projective $4$-fold.
An {\bf exceptional plane} in $X$ is a closed subset $L\subset X$
such that $L\cong\pr^2$ and
$\mathcal{N}_{L/X}\cong\mathcal{O}_{\pr^2}(-1)^{\oplus 2}$. 
Notice that if 
$C_L\subset L$ is a line, we have $-K_X\cdot C_L=1$.
An
{\bf exceptional line} in $X$ is a curve $l\cong\pr^1$ with
$\mathcal{N}_{l/X}\cong\mathcal{O}_{\pr^1}(-1)^{\oplus 3}$; notice
that $K_X\cdot l=1$.
\begin{thm}[\cite{kawsmall}]\label{flip}
Let $X$ be a projective $4$-fold and $f\colon X\dasharrow\w{X}$ a
$K$-negative
flip. Then  $\w{X}$ is smooth,
$X\smallsetminus\dom(f)$ is the disjoint union of $r\geq 1$ 
exceptional planes, and  $\w{X}\smallsetminus\dom(f^{-1})$
is the disjoint union of $r$ exceptional lines. 

Moreover $f$ factors as $h\circ g^{-1}$, where
 $g\colon \widehat{X}\to X$ is the
blow-up of $X\smallsetminus\dom(f)$, and  $h\colon\widehat{X}\to \w{X}$ is the
blow-up of $\w{X}\smallsetminus\dom(f^{-1})$.
\end{thm}
\begin{remark}\label{intersection}
Let $f\colon X\dasharrow\w{X}$ be as in Th.~\ref{flip}, $C\subset
X\smallsetminus\dom(f)$ a line in an exceptional plane, and
$l\subseteq\w{X}\smallsetminus\dom(f^{-1})$ an
 exceptional line. Let $D$ be
a divisor in $X$ and $\w{D}$ its transform in $\w{X}$. Then $D\cdot
C=-\w{D}\cdot l$. This follows easily from the factorization of $f$
as $h\circ g^{-1}$, by comparing $g^*(D)$ and $h^*(\w{D})$ in $\widehat{X}$.
\end{remark}
In this paper, our interest in Mori dream spaces is motivated by the
following result.
\begin{thm}[\cite{BCHM}, Cor.~1.3.2]
Let $X$ be a Fano manifold. Then $X$ is a Mori dream space.
\end{thm}
\noindent We are now going to explain some elementary properties of SQMs and
of rational contractions of Fano $4$-folds.
\begin{remark}\label{SQM}
Let $X$ be a Fano $4$-fold and $f\colon
X\dasharrow \w{X}$ a SQM. Then $\w{X}$ is smooth,
$X\smallsetminus\dom(f)$ is the disjoint union of $r$ 
exceptional planes, and  $\w{X}\smallsetminus\dom(f^{-1})$
is the disjoint union of $r$ exceptional lines.  

Moreover if ${C}\subset\w{X}$ is an irreducible curve such that
$C\cap\dom(f^{-1})\neq\emptyset$,  
and $C_X\subset X$ is its transform, we have 
$$-K_{\w{X}}\cdot{C}\geq -K_X\cdot C_X+s\geq 1+s\geq 1,$$
where $s$ is the number of points of ${C}$ 
which belong to an exceptional line.
In particular: 
\begin{enumerate}[(1)]
\item
if $-K_{\w{X}}\cdot{C}=1$, then ${C}$ does not
intersect any exceptional line; in general we have:
$$s\leq -K_{\w{X}}\cdot C-1;$$
\item for every irreducible curve $C\subset\w{X}$, either
  $-K_{\w{X}}\cdot C>0$ (if $C\cap\dom(f^{-1})\neq\emptyset$), or $C$
  is an exceptional line (if $C\cap\dom(f^{-1})=\emptyset$);
\item if $L\subset \w{X}$ is an exceptional plane and $l\subset \w{X}$
  is an exceptional line, then $L\cap l=\emptyset$.
\end{enumerate}
\end{remark}
\begin{proof}
The statement is trivial if $f$ is an isomorphism. 
Otherwise, let $\w{D}$ be an ample divisor in $\w{X}$, and
$D:=f^*(\w{D})$. Then $D$ is a movable divisor in $X$, and any Mori
program for $D$ yields a factorization of $f$ as a sequence of flips.
Applying \cite[Prop.~2.4]{codim}, we can 
  factor $f$ as a sequence of $m\geq 1$
$K$-negative flips.
In this way we get a factorization:
$$X\stackrel{f'}{\dasharrow}
X'\stackrel{f_m}{\dasharrow}\w{X},$$
where $f'$ is the composition of  the first $m-1$ flips, and
$f_m$ is the last one.  By induction, we can assume 
that the statement holds for $f'\colon X\dasharrow X'$.

Since $X'$ is smooth and $f_m$ is a $K$-negative flip,   
we can apply Th.~\ref{flip}; in particular, $\w{X}$ is smooth. Moreover 
$X'\smallsetminus\dom(f_m)$ is the disjoint union of $t$
exceptional planes, and $\w{X}\smallsetminus\dom(f_m^{-1})$ 
is the disjoint union of $t$ exceptional
lines.
By the induction hypothesis, an exceptional plane and an exceptional
line in $X'$ cannot meet, therefore the
indeterminacy locus of $f_m$ is disjoint from the indeterminacy
locus of $(f')^{-1}$.

We have a factorization 
$$\xymatrix{&{\widehat{X}}\ar[dl]_{g}\ar[dr]^{h}&\\
{X'}\ar@{-->}[rr]^{f_m}&&{\w{X}}}$$
where $\widehat{X}$ is smooth,
$g$ is the blow-up of $X'\smallsetminus\dom(f_m)$, and $h$ is the blow-up of 
$\w{X}\smallsetminus\dom(f_m^{-1})$. If
$E_1,\dotsc,E_t\subset\widehat{X}$ are the exceptional divisors, we
have
$$h^*(-K_{\w{X}})=g^*(-K_{X'})+\sum_{i=1}^tE_i.$$

Consider now an irreducible curve $C\subset\w{X}$ such that
$C\cap\dom(f^{-1})\neq\emptyset$, 
 and let $C_X\subset X$, $C'\subset X'$, and
$\widehat{C}\subset\widehat{X}$  be its
transforms. Suppose that $C'$ has $s'$ points belonging to an
exceptional line. Then 
$-K_{X'}\cdot C'\geq -K_X\cdot C_X+s'$ by induction, and
$C$ meets the indeterminacy locus of
$f_m^{-1}$ in
$s-s'$ points, so we get
$$-K_{\w{X}}\cdot C=-K_{X'}\cdot
C'+\sum_{i=1}^tE_i\cdot\widehat{C}\geq -K_X\cdot C_X+s'+(s-s'),$$
which gives the statement.
\end{proof}
\begin{remark}\label{ratsing}
Let $X$ be a Fano $4$-fold and $f\colon X\dasharrow Y$ a rational
contraction. Then there exists a factorization of $f$ as
$$\xymatrix{
X\ar@{-->}[r]\ar@{-->}[dr]_{f}&{\widetilde{X}}\ar[d]^{\w{f}}\\
& Y}$$
where $X\dasharrow\w{X}$ is a SQM, $\w{X}$ is smooth, 
and $\w{f}$ is a $K$-negative contraction; in
particular,  $Y$ has 
rational singularities.
\end{remark}
\begin{proof}
Consider a factorization  
$f=g_1\circ h_1$ where $h_1\colon X\dasharrow X_1$ is a
SQM 
and $g_1\colon X_1\to Y$ a
contraction.  
If $g_1$ is not $K$-negative, there exists an extremal ray $\sigma$
of $\NE({X}_1)$ such that $K_{{X}_1}\cdot \sigma\geq 0$ and
$\sigma\subseteq\NE(g_1)$. By Rem.~\ref{SQM} (2), 
$\Lo(\sigma)$ is the union of finitely many exceptional lines;
let $h_2$ be the composition of $h_1$ with the flip of
$\sigma$, and $g_2:=f\circ(h_2)^{-1}$.
$$\xymatrix{
{{X}} \ar@/^1pc/@{-->}[rr]^{h_2}\ar@{-->}[r]_{h_1}\ar@{-->}[rd]_f 
&{X_1}\ar@{-->}[r]\ar[d]^{g_1}&{X_2}\ar[dl]^{g_2}\\
&Y&}$$
Then $g_2$ is a morphism and $f=g_2\circ h_2$.
Moreover
the number of connected
components of $X\smallsetminus\dom(h_2)$ is strictly smaller than
the number of connected components of
$X\smallsetminus\dom(h_1)$. 
Proceeding in this way, after finitely many steps we get a
factorization $f=g_m\circ h_m$ where $g_m$ is $K$-negative. Finally,  $Y$ has 
rational singularities by \cite[Cor.~7.4]{kollarhigher}.
\end{proof}
\begin{corollary}\label{target}
Let $X$ be a Fano $4$-fold and $f\colon X\dasharrow Y$ a
quasi-elementary rational contraction. Then $Y$ is a Mori dream space and
moreover:
\begin{enumerate}[$\bullet$]
\item $Y$ is smooth if $\dim Y=2$;
\item $Y$ has at most isolated canonical and factorial
singularities if $\dim Y=3$. 
\end{enumerate}
\end{corollary}
\begin{proof}
The target $Y$ is a Mori dream space by Rem.~\ref{composition}.
After Rem.~\ref{ratsing} we can factor $f$ as
$X\dasharrow\w{X}\stackrel{\w{f}}{\to}Y$, where $\w{X}$ is a smooth $4$-fold, and
$\w{f}$ is a $K$-negative quasi-elementary contraction. Then the
statement follows from
\cite[Lemma~3.10]{fanos}.
\end{proof} 
\begin{corollary}\label{pluto}
Let $X$ be a Fano $4$-fold and $f\colon X\dasharrow Y$ a
quasi-elementary rational contraction. Assume that $f$ is not
regular.  

If $\dim Y=1$, then $\rho_X\leq 10$. If $\dim Y=2$, then $\rho_X\leq \rho_Y+8$.
\end{corollary}
\begin{proof}
By Rem.~\ref{ratsing}, we can factor $f$ as
$\qquad\xymatrix{
X\ar@{-->}[r]\ar@{-->}[dr]_{f}&{\widetilde{X}}\ar[d]^{\w{f}}\\
& Y}$ \\
where $X\dasharrow\widetilde{X}$ is a SQM, $\w{X}$ is smooth,
 and $\w{f}$ is a $K$-negative
quasi-elementary contraction.
The general fiber $F$ of $\w{f}$ is a smooth Fano variety, 
and $\rho_X=\rho_{\w{X}}\leq \rho_Y+\rho_F$ (see
Rem.~\ref{pippo}). 

 Since $f$ is not a morphism, $\w{X}$
contains some exceptional line $l$, which cannot intersect curves of
anticanonical degree $1$  by Rem.~\ref{SQM} (1).

We show that $F$ cannot be covered by curves
of anticanonical degree $1$. Indeed if it were,
 since $F$ is
a general fiber of $\w{f}$, we could find a (proper and irreducible)
 family of curves in $\w{X}$, covering $\w{X}$, whose general member is an
 irreducible curve, of anticanonical degree $1$, and contracted by
 $\w{f}$. As $-K_{\w{X}}$ is $\w{f}$-ample, we deduce that \emph{every} curve 
of the family has anticanonical degree $1$ and is contracted by
 $\w{f}$. On the other hand the exceptional line $l$ is not contracted by $\w{f}$,
 hence $l$ cannot be
contained in a member of the family. Thus $l$ must intersect some curve
of the family, and we get a contradiction.

If $\dim Y=2$, then $F$ is a Del Pezzo surface, thus $\rho_F\leq 9$.
Moreover if $\rho_F=9$, then $F$ 
 is covered by the pencil $|-K_F|$ which contains curves of
anticanonical degree $1$, a contradiction. Therefore $\rho_F\leq 8$
and $\rho_X\leq \rho_Y+8$. 

If $Y$ is a curve, then $F$ is a Fano $3$-fold, so that $\rho_F\leq
10$. Again, if $\rho_F=10$, then $F\cong\pr^1\times S$ where $S$ is a
Del Pezzo surface with $\rho_S=9$ (see \cite[p.~141]{fanoEMS}), and
$F$ is covered by 
curves of anticanonical degree $1$; therefore $\rho_F\leq
9$ and we get the statement. 
\end{proof}
\begin{remark}\label{torino}
Let $X$ be a Fano $4$-fold,
$f\colon X\dasharrow Y$ a quasi-elementary rational contraction, and 
$X\dasharrow \w{X}\stackrel{\w{f}}{\to}Y$ a factorization of $f$ as in 
Rem.~\ref{ratsing}. If  $D\subset Y$ is a
   non-movable prime divisor, then 
$(\w{f})^*(D)$ is a non-movable prime divisor in $\w{X}$.
\end{remark}
\begin{proof}
Let $D'\subset \w{X}$ be a prime divisor contained in the support of $(\w{f})^*(D)$. By \cite[Lemma 3.9]{fanos} we have $\w{f}(D')=D$ and $D'=
(\w{f})^*(D)$, so that $(\w{f})^*(D)$ is a prime divisor. Finally it is not difficult to check that $(\w{f})^*(D)$ is not movable.
\end{proof}
\subsection{Picard number of divisors in Fano $4$-folds}\label{pn}
Let $X$ be a Fano manifold and $D\subset X$ a prime divisor. If
$i\colon D\hookrightarrow X$ is the inclusion, let us consider 
$\N(D,X)=i_*(\N(D))\subseteq\N(X)$. We have
$\codim\N(D,X)=\dim\ker(H^2(X,\R)\to H^2(D,\R))$, therefore the
invariant $c_X$ defined in the Introduction can also be described as:
$$c_X=\max\left\{\codim\N(D,X)\,|\,D\text{ is a prime divisor in
}X\right\}.$$ 
We will need the following result.
\begin{thm}[\cite{codim}]\label{cod}
Let $X$ be a Fano $4$-fold with $c_X\geq 3$. Then one of the following holds:
\begin{enumerate}[$(i)$]
\item $X\cong S_1\times S_2$ where $S_i$ are Del Pezzo surfaces with
  $\rho_{S_1}=c_X+1\geq\rho_{S_2}$; 
\item $c_X=3$, $\rho_X=5$, and $X$ has a quasi-elementary contraction
  onto $\pr^2$;
\item   $c_X=3$, $\rho_X=6$, and $X$ has a quasi-elementary contraction
  onto $\mathbb{F}_1$ or $\pr^1\times\pr^1$. Moreover every elementary 
contraction of $X$ is either of type $(3,2)^{sm}$, or a conic bundle.
\end{enumerate}
\end{thm}
\begin{proof}
If $X\cong S_1\times S_2$ with $S_i$ Del Pezzo surfaces, we have
$c_X=\max\{\rho_{S_1}-1,\rho_{S_2}-1\}$ (see \cite[Ex.~3.1]{codim}),
so up to exchanging $S_1$ and $S_2$ we get $\rho_{S_1}=c_X+1\geq\rho_{S_2}$.

If $X$ is not a product of surfaces, then by \cite[Cor.~1.3 and
  Th.~3.3]{codim} we have $c_X=3$, $\rho_X\leq 6$, and
$X$ has a
quasi-elementary contraction $f\colon X\to S$ where $S$ is
 a smooth Del Pezzo surface with
$\rho_X-\rho_S=4$. Thus
$\rho_S\in\{1,2\}$, and if $\rho_S=1$ we get $(ii)$. 

Suppose that
$\rho_S=2$, and let $g$ be an elementary contraction of $X$.
If $\NE(g)\not\subset\NE(f)$, then $f$ is finite on every non-trivial 
fiber $F$ of $g$. Since $\dim\N(F,X)=1$, we cannot have $f(F)=S$,
therefore $\dim F\leq 1$. Hence $g$ is either
 of type $(3,2)^{sm}$, or a conic bundle, by Th.~\ref{smallfibers}. 

Suppose that $\NE(g)\subset\NE(f)$.
After \cite[proof of Prop.~3.3.1, in particular \S3.3.15]{codim}, $f$ factors as
$h_2\circ h_1$ where $h_1\colon X\to Y$ and $h_2\colon Y\to S$ are conic
bundles, $Y$ is smooth with $\dim Y=3$ and $\rho_Y=3$, and $\NE(h_1)$
contains $4$ extremal rays, all of type $(3,2)^{sm}$. Therefore either
$\NE(g)\subset\NE(h_1)$ and we are done, or $(h_1)_*(\NE(g))=\NE(h_2)$. In
this last case,  
if $F$ is  a non-trivial fiber of $g$, then $h_1$ is
finite on $F$ and $h_1(F)$ is contained in a fiber of $h_2$, therefore
$\dim F=1$ and we get the statement. 
\end{proof}
Therefore, if we are interested in studying Fano $4$-folds  which are
not products and have large
Picard number, we can assume that $c_X\leq
2$, so that for every prime divisor $D\subset X$ we have
$\dim\N(D,X)\geq \rho_X-2$. 
Let us also state the following application.
\begin{corollary}\label{eco}
Let $X$ be a Fano $4$-fold. If $X$ has a small elementary
contraction\footnote{This is equivalent to
 $\Mov(X)\supsetneq\Nef(X)$.} then either $\rho_X=5$
and $c_X=3$, or $c_X\leq 2$.
\end{corollary}
It is natural to ask whether we can
deduce similar properties for a SQM of $X$. The following two
statements describe how $\dim\N(D,X)$ varies under a flip or a SQM.
\begin{remark}\label{divisors}
Let $X$ be a smooth $4$-fold, $\sigma$ a $K$-negative small extremal ray
of $\NE(X)$, $X\dasharrow \w{X}$ the flip of $\sigma$, and $\w{\sigma}$ the
corresponding small extremal ray of $\NE(\w{X})$.
\begin{enumerate}[(1)]
\item Let $Z\subset X$ be a closed subset disjoint from $\Lo(\sigma)$, and
  $\w{Z}\subset\w{X}$ its transform. Then
  $\dim\N(\w{Z},\w{X})=\dim\N(Z,X)$. 
\item
Let $D\subset{X}$ be a prime divisor, and
$\w{D}\subset\w{X}$ its transform. 
Then either $\dim\N(\w{D},\w{X})=\dim\N(D,{X})$, or
$\dim\N(\w{D},\w{X})=\dim\N(D,{X})-1$. If the last case occurs,
then $D\cdot \sigma<0$, $\w{D}\cdot\w{\sigma}>0$, and
$\w{\sigma}\not\subset\N(\w{D},\w{X})$.  
\end{enumerate}
\end{remark}
\begin{proof}
We have the standard flip diagram:
$\quad\xymatrix{
X\ar[dr]_g\ar@{-->}[rr]&&{\w{X}}\ar[dl]^{\w{g}}\\
&Y&
}$\\
where $g$ and $\w{g}$ are the contractions of $\sigma$ and
$\w{\sigma}$ respectively.

\noindent(1)\ \  We have $g(Z)=\w{g}(\w{Z})$ and
$g_*(\N(Z,X))=\N(g(Z),Y)=\w{g}_*(\N(\w{Z},\w{X}))$. 

We show that $\sigma\subset\N(Z,X)$ if and only if
$\w{\sigma}\subset\N(\w{Z},\w{X})$. Indeed let $B\subset \Lo(\sigma)$ be a line
in an exceptional plane, and $l\subseteq \Lo(\w{\sigma})$ an exceptional
line. If  $\sigma\subset\N(Z,X)$, then  
$B\equiv\sum_i\lambda_i C_i$, with $\lambda_i\in\Q$ and $C_i\subset Z$
irreducible curves. Let $\w{C}_i\subset\w{Z}$ be the transform of
$C_i$. 
Then there exists $\mu\in\Q$ such that $\mu l\equiv\sum_i\lambda_i
\w{C}_i$. On the other hand, by taking anticanonical degrees, we get 
$$1=-K_X\cdot B=\sum_i\lambda_i (-K_X)\cdot
C_i=\sum_i\lambda_i(-K_{\w{X}})\cdot\w{C}_i=-\mu,$$ 
therefore $\mu\neq 0$ and $[l]\in\N(\w{Z},\w{X})$. The other
implication is shown in the same way. 

Therefore $\ker g_*\subseteq \N(Z,X)$ if and only if $\ker
\w{g}_*\subseteq\N(\w{Z},\w{X})$, which yields
$\dim\N(Z,X)=\dim\N(\w{Z},\w{X})$.  

\medskip

\noindent(2)\ \  If $\Lo(\sigma)\cap D=\emptyset$, then 
$\dim\N(D,X)=\dim\N(\w{D},\w{X})$ by (1).  
Suppose that $\Lo(\sigma)\cap D\neq \emptyset$. 
By Th.~\ref{flip}, $\Lo(\sigma)$ is a union of exceptional planes, in
particular there is a curve $C\subseteq\Lo(\sigma)\cap D$. Then
$[C]\in\sigma\cap\N(D,X)$, so that
 $\sigma\subset\N(D,X)$, and we get:
$$\dim\N(D,X)=\begin{cases}
\dim\N(\w{D},\w{X})\quad&\text{if }\w{\sigma}\subset \N(\w{D},\w{X});\\
\dim\N(\w{D},\w{X})+1\quad&\text{if }\w{\sigma}\not\subset \N(\w{D},\w{X}).
\end{cases}$$
In the last case, $\Lo(\w{\sigma})\cap\w{D}$ must be a (non-empty)
finite set, therefore we
have
$\w{D}\cdot \w{\sigma}>0$ and $D\cdot \sigma<0$.
\end{proof}
\begin{corollary}\label{dimension}
Let $X$ be a Fano $4$-fold, $f\colon X\dasharrow\w{X}$ a SQM, 
$D\subset X$ a prime divisor, and $\w{D}\subset\w{X}$ its transform.
\begin{enumerate}[$\bullet$]
\item If $f$ factors as a sequence of $m$ $K$-negative
flips, then $\dim\N(D,X)\leq\dim\N(\w{D},\w{X})+m$. 
\item If $D$
does not contain 
exceptional planes, then $\dim\N(D,X)=\dim\N(\w{D},\w{X})$. 
\end{enumerate}
\end{corollary}
\subsection{Characterization of non-movable prime divisors}\label{nm}
In this section we give 
a geometric description of non-movable prime divisors
in a Fano $4$-fold $X$ with $\rho_X\geq 6$ (Th.~\ref{effective}). As noticed in
Rem.~\ref{nonmovable}, the classes of these  divisors generate the
one-dimensional faces of $\Eff(X)$ which do not lie in $\Mov(X)$. On
the other hand, we show  that if a
one-dimensional face of $\Eff(X)$ is contained in $\Mov(X)$, then
$\rho_X\leq 11$ 
(Prop.~\ref{lorenzo}). Th.~\ref{effective} also allows to describe elementary
divisorial rational contractions of $X$, see Cor.~\ref{monaco}.

 We refer the reader to \cite{sammy} for a study of
$\Eff(X)$ and $\NM(X)$ for a Fano $4$-fold $X$.
\begin{thm}\label{effective}
Let $X$ be a Fano $4$-fold with $\rho_X\geq 6$, and
 $D\subset X$ a non-movable prime divisor. Then
there exists a diagram: 
$$\xymatrix{
X\ar@{-->}[r]&{\widetilde{X}}\ar[d]^{f}\\
& Y}$$
where $X\dasharrow\w{X}$ is a SQM whose indeterminacy locus is the union
of exceptional planes $L$ such that $D\cdot C_L<0$ for a line
$C_L\subset L$, and
$f$ is an elementary divisorial
 contraction with $\Exc(f)=\w{D}$ (the transform of $D$). Moreover
 one of the following
holds:
\begin{enumerate}[$(i)$]
\item $X=\w{X}$, $D$ is the locus of an extremal ray of type $(3,2)$,
  and $D$ does not contain any exceptional plane;
\item 
$Y$ is smooth and Fano, $f$ is
the blow-up of a smooth curve, and $\w{D}$ is a
$\pr^2$-bundle over a smooth curve;
\item 
$Y$ is smooth and Fano, $f$ is
the blow-up of a point, and $\w{D}\cong\pr^3$;
\item 
$\w{D}$ is isomorphic to a
quadric, $f(\w{D})$ is a factorial and terminal singular point, 
and $Y$ is Fano.
\end{enumerate}
\end{thm}
\noindent We will say that $D$ is              \emph{of type $(3,2)$, of
type $(3,1)$, of type $(3,0)^{\pr^3}$,} or \emph{of type $(3,0)^Q$}, when we are
respectively in case $(i)$, $(ii)$, $(iii)$, or $(iv)$ above. 
\begin{remark}
In cases $(ii)-(iv)$ we will also show that $c_X\leq 2$, and that the
birational map $X\dasharrow\w{X}$ factors as a sequence of at least
$\rho_X-4$ $D$-negative and $K$-negative
flips (this follows from \eqref{oggi}). In particular,
Th.~\ref{effective} implies Th.~\ref{secondo}. 
\end{remark}
\begin{example}[A non-movable prime divisor  of type $(3,0)^{\pr^3}$]
Let $Y:=(\pr^1)^4$ and let $f\colon\w{X}\to Y$ be the blow-up of a point $p\in
Y$. 
Then $\w{X}$ is a toric $4$-fold with $\rho_{\w{X}}=5$; in
particular $\w{X}$ is a Mori dream space. 

Let $C_1,C_2,C_3,C_4\subset Y$ be the irreducible curves of type
$\pr^1\times\{pts\}$ through $p$, and $l_i\subset\w{X}$ the transform
of $C_i$. We have $-K_Y\cdot C_i=2$, $-K_{\w{X}}\cdot l_i=-1$, and
$\Exc(f)\cdot l_i=1$; in
particular, $\w{X}$ is not Fano.

On the other hand $l_i$ is an exceptional line, $\R_{\geq 0}[l_i]$ is
a small extremal ray of $\NE(\w{X})$, and it is possible to flip these
exceptional lines with a sequence of $4$ flips $\w{X}\dasharrow X$.

Then $X$ is Fano\footnote{This toric Fano $4$-fold is described in
  \cite[Prop.~3.5.8(iii)]{bat2}.} 
and the transform $D\subset X$ of $\Exc(f)$ is a smooth
divisor, isomorphic to the blow-up of $\pr^3$ in $4$ points. 
There are $4$
exceptional planes $L_1,\dotsc,L_4\subset D$, 
and $D\cdot C_{L_i}=-1$ where $C_{L_i}\subset L_i$ is a line. 
\end{example}
\begin{proof}[Proof of Th.~\ref{effective}]
After \cite[Prop.~2.4]{codim}, there exists a Mori program for $D$ such that every extremal ray of the program is $K$-negative. By Rem.~\ref{nonmovable}, this gives 
 a SQM $g\colon X\dasharrow\w{X}$, which
factors as a sequence of $D$-negative and $K$-negative flips, and an elementary,
$K$-negative, divisorial  contraction $f\colon \w{X}\to Y$ with
exceptional divisor $\w{D}$, the transform of $D$. Notice that
$\w{D}$ has positive
intersection with all exceptional lines in $\w{X}$.

If $f$ is of type $(3,2)$, then $\widetilde{D}$ is covered by curves
of anticanonical degree $1$, thus by Rem.~\ref{SQM} (1) $\w{D}$ cannot 
intersect any exceptional
line. This means that $X=\w{X}$ and $D$ is the locus of the extremal
ray $\NE(f)$, of type $(3,2)$.

Suppose that $D$ contains an exceptional plane $L$. 
After the classification of possible isolated $2$-dimensional fibers
of $f$ in \cite[Th.~4.7]{AWaview}, we know that an exceptional
plane cannot be a component of a fiber of $f$, therefore $f$ is finite
on $L$. Thus $f(L)=f(D)$, which implies that $\dim\N(f(D),Y)=1$. On
the other hand $\N(D,X)=(f_*)^{-1}\N(f(D),Y)$, so that
$\dim\N(D,X)=2$, and $c_X\geq \rho_X-2\geq 4$.
Then $X$ should be a product of
surfaces by Th.~\ref{cod}, thus $X$ cannot contain any exceptional
plane, and we get a contradiction. Therefore we 
have $(i)$.

\medskip

Suppose that $f$ is not of type $(3,2)$. Since $\rho_X\geq 6$, by
\cite[Cor.~1.3]{31} $X$ cannot have  
elementary divisorial
contractions of type $(3,0)$ or $(3,1)$, therefore $g$
is not an isomorphism and $X$ has a small
elementary contraction. Hence $c_X\leq 2$ by Cor.~\ref{eco}, 
and $\dim\N(D,X)\geq 4$. 

Let
$l_1,\dotsc,l_r\subset\w{X}$ be the exceptional lines, and suppose
that $g$ factors as a sequence of $m\geq 1$ $K$-negative and
$D$-negative flips. Then Rem.~\ref{dimension} yields
$m\geq \dim\N(D,X)-\dim\N(\w{D},\w{X})$, therefore:
\stepcounter{thm}
\begin{equation}
\label{oggi} 
r\geq m\geq 
\rho_X-c_X-\dim\N(\w{D},\w{X})\geq
4-\dim\N(\w{D},\w{X}).
\end{equation} Moreover
$\w{D}\cdot l_i>0$ for every $i=1,\dotsc,r$, and $l_i$
can not be contained in a fiber of $f$.

\medskip

Suppose that $f$ is of type $(3,1)$, so that $\dim\N(\w{D},\w{X})=2$
and $r\geq 2$ by \eqref{oggi}.
 Let $F\subset\w{D}$ be a fiber of $f$ 
intersecting  $l_1\cup\cdots\cup l_r$. 
Then $F$ cannot be covered by curves of
anticanonical degree $1$ by Rem.~\ref{SQM} (1).
By the classification of elementary $K$-negative contractions
of type $(3,1)$ in \cite{takagi}, this is possible only if
  $Y$ is smooth and $f$ is
the blow-up a smooth curve $C\subset Y$, so that $\w{D}$ is a
$\pr^2$-bundle over $C$. Moreover the lines in the fibers of
$f_{|\w{D}}$
 have anticanonical degree $2$
in $\w{X}$.

If $F$
intersects $l_1\cup\cdots\cup l_r$
in at least two points, by taking the line through these two points we get a
contradiction with Rem.~\ref{SQM} (1).
Therefore every fiber of $f_{|\w{D}}$ intersects $l_1\cup\cdots\cup
l_r$ in at most one point, and 
the exceptional lines
$l_1,\dotsc,l_r$ 
intersect different fibers of $f_{|\w{D}}$. 
Since $r\geq 2$, this implies that no exceptional line is contained
in $\w{D}$.

Let's show that $Y$ is Fano. We have $f^*(-K_Y)=-K_{\w{X}}+2\w{D}$
and $-K_Y\cdot f(l_i)=(-K_{\w{X}}+2\w{D})\cdot l_i=2\w{D}\cdot l_i-1>0$ 
for every $i=1,\dotsc,r$. 
If $\sigma$ is an extremal ray of $\NE(\w{X})$ with $-K_{\w{X}}\cdot \sigma>0$
and $\w{D}\cdot \sigma\geq 0$, then 
$(-K_{\w{X}}+2\w{D})\cdot \sigma>0$.

Suppose that $\w{X}$ has a $\w{D}$-negative
 extremal ray $\sigma\neq\NE(f)$.
Then $\Lo(\sigma)\subseteq\w{D}$, so that $-K_{\w{X}}\cdot \sigma>0$. If
$G\subset\w{D}$ is a non-trivial fiber of the contraction of $\sigma$, then
$f$ must be finite on $G$, hence $\dim G=1$. Therefore $\sigma$
is of type $(3,2)$ (see Th.~\ref{smallfibers}),
and $\w{D}$ is
covered by curves of anticanonical degree $1$, a
contradiction by Rem.~\ref{SQM} (1). We deduce that
 $-K_{\w{X}}+2\w{D}$ is nef and
$(-K_{\w{X}}+2\w{D})^{\perp}\cap\NE(\w{X})=\NE(f)$, hence $-K_Y$ is
ample and we have $(ii)$.
 
\medskip

Assume now that $f$ is of type $(3,0)$, so that
$\dim\N(\w{D},\w{X})=1$ and $r\geq 3$ by \eqref{oggi}.

Suppose that $\w{D}\cong\pr^3$. Since $-K_{\w{X}}\cdot\NE(f)>0$, we
have $\mathcal{N}_{\w{D}/\w{X}}\cong\mathcal{O}_{\pr^3}(-a)$ with
$a\in\{1,2,3\}$. 
If $a=3$, then $\w{D}$ is covered by curves of anticanonical degree
$1$, which is impossible by Rem.~\ref{SQM} (1), because $\w{D}$
intersects $l_1$. 
If $a=2$, the lines in $\w{D}$ have anticanonical degree $2$ in $X$,
and by taking a line which intersects both $l_1$ and $l_2$, we get
again
a contradiction by Rem.~\ref{SQM} (1). 
Therefore $a=1$, $Y$ is smooth, and $f$ is the blow-up of a point $p\in Y$. 

We have $f^*(-K_Y)=-K_{\w{X}}+3\w{D}$ and $-K_Y\cdot f(l_i)=
(-K_{\w{X}}+3\w{D})\cdot
l_i=3\w{D}\cdot l_i-1>0$ 
for every $i=1,\dotsc,r$, and similarly as before we conclude
that $Y$ is Fano, so we get $(iii)$.

\medskip

Suppose that $\w{D}\cong Q$, where $Q\subset\pr^4$ is a
quadric. Again we have
$\mathcal{N}_{\w{D}/\w{X}}\cong\mathcal{O}_Q(-a)$ with $a\in\{1,2\}$. If
$a=2$, then $\w{D}$ is covered by curves of anticanonical degree $1$,
which is impossible. Thus $a=1$, and if $C\subset \w{D}$ corresponds to a
line in $Q$, we
have $-K_{\w{X}}\cdot C=2$ and $\w{D}\cdot C=-1$. The point
$p=f(\w{D})\in 
Y$ is a factorial terminal singularity in $Y$, and
$f^*(-K_Y)=-K_{\w{X}}+2\w{D}$. As before we see that
$-K_Y\cdot f(l_i)=2\w{D}\cdot l_i-1>0$ for every $i=1,\dotsc,r$,
and $Y$ is Fano, so we get $(iv)$.
  
\medskip

We assume now that $\w{D}$ is not isomorphic to $\pr^3$ or a quadric,
and show that 
this gives a contradiction.
This type of exceptional divisor has been studied
by Beltrametti \cite{beltra,beltrametti86} and by Fujita 
as an application of his theory
of Del Pezzo varieties -- we refer the reader to
\cite[\S3.2]{fanoEMS} for an overview.

 Notice that $\w{D}$ is reduced and irreducible. Being a
 divisor in a smooth variety, it is Cohen-Macaulay and has a
 locally free dualising sheaf $\omega_{\w{D}}$ given by
$$\omega_{\w{D}}=\mathcal{O}_{\w{X}}(K_{\w{K}}+\w{D})_{|\w{D}}.$$
 Therefore $\w{D}$ is
 Gorenstein, and by Serre's criterion, it is normal if and only if
$\dim\Sing\w{D}\leq 1$.

By \cite[\S3, in particular (3.2)]{fuji2}, 
 there exists an ample line bundle $L_{\w{X}}\in\Pic(\w{X})$
such that, if $L:=(L_{\w{X}})_{|\w{D}}$, we have
$$\mathcal{O}_{\w{X}}(-K_{\w{X}})_{|\w{D}}=\mathcal{O}_{\w{X}}(-\w{D})_{|\w{D}}= 
L,$$
 hence
$\omega_{\w{D}}=L^{\otimes (-2)}$. Moreover 
the pair $(\w{D},L)$ is a Del Pezzo variety, see
\cite[(3.3)]{fuji2} and \cite[\S3.2]{fanoEMS} for the definition.

Notice that 
$\w{D}$ cannot be covered by curves having intersection $1$ with
$L$, because these would have anticanonical degree $1$ in $\w{X}$,
contradicting Rem.~\ref{SQM} (1). 

Set $d:=L^3$. 
If $d=1$, then by \cite[Th.~3.2.5 (i)]{fanoEMS}
$\w{D}$ is isomorphic to a 
hypersurface of degree $6$ in the weighted projective space
$\pr(3,2,1,1,1)$. 

Since $\pr(3,2,1,1,1)$ has two singular points\footnote{This can be
  seen for instance using toric geometry, see
  \cite[Th.~3.6]{conrads}.}, 
the generic
hypersurface is smooth; if in the smooth case $\w{D}$ is covered by curves having
intersection $1$ with 
$L$, the same must hold also in the singular case.

Hence suppose that 
$\w{D}$ is
smooth. By \cite[Prop.~3.2.4 (i)]{fanoEMS}  
the general element $S\in |L|$ is a smooth
surface
with
$-K_S=L_{|S}$ ample and
$(-K_S)^2=d=1$. Therefore $S$ is
 covered by curves of anticanonical degree $1$ (the
pencil $|-K_S|$)
 and
$\w{D}$ is covered by curves having
intersection $1$ with 
$L$, which gives a contradiction.

\medskip

If $d=2$, then by \cite[Prop.~3.2.4~(ii)]{fanoEMS} the linear system 
$|L|$
determines a double covering
$\pi\colon\w{D}\to\pr^3$
such that $L=\pi^*\mathcal{O}_{\pr^3}(1)$. For $i=1,2$ choose
$p_i\in\w{D}\cap l_i$, and let $C\subset\pr^3$ be a line through
$\pi(p_1)$ and $\pi(p_2)$. Set $C':=\pi^{-1}(C)\subset\w{D}$. Then
$p_1,p_2\in C'$, $\pi_*(C')=2C$ and 
$-K_{\w{X}}\cdot C'=(L\cdot C')_{\w{D}}=2$ (where
$(\ \cdot\ )_{\w{D}}$ is intersection in $\w{D}$). The curve $C'$ can
not be 
irreducible by Rem.~\ref{SQM} (1), 
but if it is reducible we get a curve of anticanonical
degree $1$ in $\w{X}$ containing one of the points $p_i$, which is
again impossible.

\medskip

Suppose now that $d\geq 3$. Then $L$ is very ample and gives
an isomorphism of $\w{D}$ with $V\subset \pr^{d+1}$ of degree $d$, see
\cite[Prop.~3.2.4~(ii)]{fanoEMS}.

If $d=3$ then $V$ is a cubic in $\pr^4$, thus it is covered by lines,
and  $\w{D}$ is covered by curves having
intersection $1$ with 
$L$. Similarly, if $d=4$, then 
by \cite[Th.~3.2.5~(iv)]{fanoEMS} 
$V$ is
the complete intersection of two
quadrics in $\pr^5$, and again it is covered by lines.

\medskip

Assume that $d\geq 5$. Then by \cite[(2.6)]{fuji2}
$V\subset\pr^{d+1}$ 
is not a cone over another variety.

If  $\w{D}$ is smooth, then it is a
 Fano $3$-fold of index
$2$, 
and by  \cite[Th.~3.3.1]{fanoEMS}
the possibilities for $\w{D}$ are: the blow-up of $\pr^3$ in a point,
$(\pr^1)^3$, $\pr_{\pr^2}(T_{\pr^2})$, or a linear section of
$G(1,4)\subset\pr^9$. 
 In all these cases it is
easy to see that $\w{D}$ is covered by curves having
intersection $1$ with 
$L$.

Suppose now that $\dim\Sing(\w{D})=0$. Then $\w{D}$ is normal, 
and by \cite[Th.~(2.1) and (2.9)]{Fujita86} we see that the singularities of
$\w{D}$ are ordinary double points; in particular $\w{D}$ has terminal
singularities. Therefore by \cite[Th.~11]{namikawa} $\w{D}$ has a
smoothing $T$, where $T$ is a smooth Fano $3$-fold with index $2$ and
anticanonical degree $8d$. By the previous case, we know that $T$ is
covered by curves of anticanonical degree $2$, hence the same holds
for $\w{D}$.

If $\dim\Sing(\w{D})=2$ then $\w{D}$ is not normal, and by
\cite[Th.~(2.1)]{Fujita86} $V$ is the projection of a smooth
variety of minimal degree in $\pr^{d+2}$. In particular $V$ is covered
by lines, and we are done.

If instead $\dim\Sing(\w{D})=1$ (so that $\w{D}$ is normal), 
we follow the construction in
\cite[(6), p.~150]{Fujita86}. Let $p_0\in \Sing(V)$, and set
$$W:=\overline{\bigcup_{q\in V\smallsetminus p_0}
  \overline{qp_0}}\subset\pr^{d+1},$$
where $\overline{pq}$ denotes the line through $p$ and $q$ in
$\pr^{d+1}$.  Notice that $\dim W=4$ and $W$ has degree $d-2$.
Set moreover
$$R:=\left\{p\in W\,|\,\overline{pq}\subset W\text{ for every }q\in W\smallsetminus
p\right\}$$ 
(so that $p_0\in R$).
By \cite[Lemma~(2)]{Fujita86},
 $R\subset\pr^{d+1}$ is a linear subspace, and if $M\subset W$ is a
section of $W$ with a generic linear subspace of dimension $d-\dim R$, 
then $W$ is the cone over
$M$ with vertex $R$.
By \cite[(6)]{Fujita86} $M$ is smooth and $R\subseteq\Sing(V)$,
therefore $\dim R\in\{0,1\}$. 

All the possibilities for $V$ are
listed in \cite[Th.~(2.9)]{Fujita86}. Since $\dim V=3$ and
$\dim\Sing(V)=1$, we see that the possibilities are: (vi), (si22i),
(si31i), (si211), (si21i-a), (si111o-d), and (si21i-b). In the cases (vi),
(si22i), (si31i), (si21i-a), and (si21i-b) we have $\dim R=1$ (see
\cite{Fujita86}, pages
155, 169, 170, and 163 respectively). 

In case (si111o-d) we have
$\dim R=0$, however this variety $V$ is the same as (si21i-a), see 
\cite[Remark on p.~167]{Fujita86}. By choosing the point
$p_0$ in a
one-dimensional component of $\Sing(V)$, we can reduce to the case
where $\dim R=1$. The case (si211) is analogous, see \cite[Remark on
  p.~171]{Fujita86}. 

Therefore $R$ is a line and $\dim M=2$. We still follow the
construction in \cite[(7)]{Fujita86}. Let $\w{P}\to\pr^{d+1}$ be the
blow-up along $R$, 
let $\w{V}\subset\w{P}$ and $\w{W}\subset\w{P}$ be transforms of $V$
and $W$ respectively,
and $\ph\colon\w{W}\to W$
 the induced morphism.

 Then $\w{W}$ is smooth and there is a 
 $\pr^2$-bundle structure $\w{W}\to M$ such that if
 $F\subset\w{W}$ is a fiber we have
$\ph^*(\mathcal{O}_W(1))_{|F}=\mathcal{O}_{\pr^2}(1)$.
On the other hand by \cite[(8)]{Fujita86} we also have
$\ph^*(\mathcal{O}_W(1))_{|F}= 
\mathcal{O}_{\w{W}}(\w{V})_{|F}$, therefore for a generic $F$ the
intersection 
$\w{V}\cap F$ is a line in $F$, and $\ph(\w{V}\cap F)$ is a line in
$V\subset\pr^{d+1}$.
This shows that $V$ is covered by lines, and concludes the proof.
\end{proof}
\begin{corollary}[Elementary divisorial rational contractions]\label{monaco}
Let $X$ be a Fano $4$-fold with $\rho_X\geq 6$, and consider an
elementary divisorial contraction $f\colon \w{X}\to Y$, where
$\w{X}$ is  a SQM
of $X$. Then $Y$ has at most isolated terminal and factorial singularities.
Moreover one of the following holds:
\begin{enumerate}[$(i)$]
\item $f$ is of type $(3,2)$, $X\dasharrow \w{X}$ is an isomorphism
  over $\Exc f$, and $\Exc(f)$ does not contain any exceptional plane;
\item 
$Y$ is smooth and it is a SQM of a Fano $4$-fold, and $f$ is the blow-up of a
smooth curve $C\subset Y$; moreover
 if $C_0\subset Y$ is an irreducible
curve with $C\cap C_0\neq\emptyset$ and $C_0\neq C$, then $-K_Y\cdot
C_0\geq 3$, except possibly for finitely many exceptions where $-K_Y\cdot
C_0=1$, $\Exc(f)\cdot\w{C}_0=1$, and $-K_{\w{X}}\cdot \w{C}_0=-1$
($\w{C}_0\subset\w{X}$ the transform of $C_0$); 
\item 
$Y$ is smooth and it is a SQM of a Fano $4$-fold, and
 $f$ is blow-up of a
 point $p\in Y$; moreover if $C_0\subset Y$ is an irreducible
curve with $p\in C_0$, then $-K_Y\cdot
C_0\geq 4$, except possibly for 
finitely many exceptions where $-K_Y\cdot
C_0=2$, $\Exc(f)\cdot\w{C}_0=1$, and $-K_{\w{X}}\cdot \w{C}_0=-1$;
\item 
$\Exc(f)$ is isomorphic to an
irreducible quadric and $p:=f(\Exc(f))$ is a point; 
moreover if $C_0\subset Y$ is an irreducible
curve with $p\in C_0$, then $-K_Y\cdot
C_0\geq 3$, except possibly for finitely many exceptions where $-K_Y\cdot
C_0=1$, $\Exc(f)\cdot\w{C}_0=1$, and $-K_{\w{X}}\cdot \w{C}_0=-1$.
\end{enumerate}
\end{corollary}
\begin{proof}
Let $D\subset X$ be the transform of $\Exc(f)$. Then $D$ is a
non-movable  prime divisor, and 
 by Th.~\ref{effective}
there is a diagram
$$\xymatrix{X\ar@{-->}[r]&{\w{X}_1}\ar@{-->}[r]\ar[d]_{f_1}
  &{\w{X}}\ar[d]^{f}\\
& Y_1\ar@{-->}[r]&Y}$$
where $X\dasharrow\w{X}_1$ is a SQM and $f_1\colon \w{X}_1\to Y_1$ is 
an elementary divisorial contraction with exceptional divisor the
transform of $\Exc(f)$, and satisfying one of the conditions of Th.~\ref{effective}. 
The birational map $Y_1\dasharrow Y$ is an isomorphism in
codimension $1$, \emph{i.e.}\ it is a SQM.

The cases $(i)$ - $(iv)$ of the statement correspond to the same cases
of Th.~\ref{effective};
we will consider 
$(i)$ and $(ii)$, the other ones being completely analogous.

If $D$ is of type $(3,2)$, then $X=\w{X}_1$, $f_1$ is an elementary
contraction of type 
$(3,2)$, and $D$ does not contain any exceptional plane; in particular 
$\dim\N(D,X)=\dim\N(\Exc(f),\w{X})$ by Rem.~\ref{dimension}.

If the map $X\dasharrow\w{X}$ is not an isomorphism, then
Cor.~\ref{eco} yields that $c_X\leq 2$. Hence
$$\dim\N(\Exc(f),\w{X})=\dim\N(D,X)\geq\rho_X-2\geq 4,$$
and $f$ cannot be of type $(3,0)$ nor $(3,1)$. Therefore $f$ is of
type $(3,2)$ and $\Exc(f)$ is covered by curves of anticanonical
degree $1$. By Rem.~\ref{SQM} (1) the map $X\dasharrow\w{X}$ is an isomorphism
over $\Exc(f)$, and  we get $(i)$.

\medskip

Suppose now that $D$ is of type $(3,1)$. Then $Y_1$ is smooth and
Fano, so that the birational map $Y_1\dasharrow Y$ is an isomorphism
outside a disjoint union  of exceptional planes in $Y_1$, see
Rem.~\ref{SQM}. Moreover
 $f_1$ is the blow-up of a smooth curve $C_1\subset
Y_1$, and
$(f_1)^*(-K_{Y_1})=-K_{\w{X}_1}+2\Exc(f_1)$.

Let $l_1,\dotsc,l_r\subset\w{X}_1$ be the exceptional lines. Then
$f_1(l_i)$ intersects $C_1$ and $-K_{Y_1}\cdot f(l_i)=2\Exc(f_1)\cdot
l_i-1$, hence $-K_{Y_1}\cdot f(l_i)\geq 3$ unless $-K_{Y_1}\cdot
f(l_i)=\Exc(f_1)\cdot 
l_i=1$.
On the other hand let
 $C_2\subset Y_1$ be an irreducible curve  different from 
$C_1,f_1(l_1),\dotsc,f_1(l_r)$. If
 $C_1\cap C_2\neq\emptyset$, then 
$-K_{Y_1}\cdot C_2\geq 3$.

Let now $L\subset Y_1$ be an exceptional plane. Since
 $C_1$ can intersect at most finitely many curves of
anticanonical degree $1$, we have $C_1\cap L=\emptyset$, and
$f_1^{-1}(L)\subset\w{X}_1$ is still an exceptional plane. Then
$(l_1\cup\cdots\cup l_r)\cap f_1^{-1}(L)=\emptyset$ by Rem.~\ref{SQM} (3),
thus $(f_1(l_1)\cup\cdots\cup f_1(l_r))\cap L=\emptyset$.

 We conclude that $C_1\cup
f_1(l_1)\cup\cdots\cup f_1(l_r)$ is contained in the open subset
where $Y_1\dasharrow Y$ is an isomorphism, and $\Exc(f_1)\cup
l_1\cup\cdots\cup l_r$ is contained in the open subset
where $\w{X}_1\dasharrow \w{X}$ is an isomorphism.

Therefore $f$ is the blow-up of a smooth curve $C\subset Y$, and $C$
does not meet any exceptional line in $Y$. Let $C_0\subset Y$ be an
irreducible curve which meets $C$, $C_0\neq C$, and let $C_0'\subset
Y_1$ be its transform. We have $-K_Y\cdot C_0\geq -K_{Y_1}\cdot C_0'$
by Rem.~\ref{SQM}, so that either  $-K_Y\cdot C_0\geq 3$, or
$C_0'=f_1(l_i)$ for some $i\in\{1,\dotsc,r\}$ and
 $-K_{Y_1}\cdot C_0'=\Exc(f_1)\cdot
l_i=1$; this gives the statement.
\end{proof}
We conclude this section showing that when the
cones $\Mov(X)$ and $\Eff(X)$ share a one-dimensional face, we can
easily bound the Picard number of $X$. As a consequence, when $\rho_X$
is large, $X$ contains plenty of non-movable prime divisors.
\begin{proposition}\label{lorenzo}
Let $X$ be a Fano $4$-fold, and 
suppose that there exists a movable prime divisor $D$ whose class
belongs to a one-dimensional face of $\Eff(X)$.   

Then $\rho_X\leq 11$. Moreover if $\rho_X=11$ then 
 $D$ is a fiber
 of a quasi-elementary contraction $X\to\pr^1$, with general fiber 
 $\pr^1\times S$, 
 where $S$ is a Del Pezzo surface with $\rho_S=9$.
\end{proposition}
\begin{proof}
The cone $\R_{\geq 0}[D]$ is a common  one-dimensional face of
 $\Mov(X)$ and $\Eff(X)$. By Cor.~\ref{tobia}, this implies the
existence of a quasi-elementary rational contraction
 $f\colon X\dasharrow Y$ with
$\rho_Y=1$. 
 
If $\dim Y=3$, then $f$ is elementary and $\rho_X=2$ (see Rem.~\ref{pippo}). 

Assume that $\dim Y\leq 2$.
If $f$ is not regular, then $\rho_X\leq 10$ by Cor.~\ref{pluto}. 
If $f$ is regular, the statement follows as in the proof of
Cor.~\ref{pluto}.
\end{proof}
\begin{corollary}
Let $X$ be a Fano $4$-fold with $\rho_X\geq 12$. 
Then $\Eff(X)$ is generated by classes of 
non-movable prime
divisors; in particular $X$ contains al least $\rho_X$
such divisors.
\end{corollary}
\section{Rational contractions of fiber type of Fano $4$-folds}\label{rcft}
\subsection{Quasi-elementary rational contractions onto
  surfaces}\label{surf}
In this section we study Fano $4$-folds having a quasi-elementary
rational contraction onto a surface. 
First of all let us recall what happens in the case of a regular contraction.
\begin{proposition*}[\cite{fanos}, Th.\ 1.1 (i)]
Let $X$ be a Fano $4$-fold and $f\colon X\to S$ a quasi-elementary
contraction onto a surface.
Then $\rho_S\leq 9$,  $\rho_X\leq 18$, and $\rho_X=18$ only if $X$ is
a product of surfaces.

If $f$ is elementary, then $\rho_X\leq 10$, with equality only
if $X\cong\pr^2\times S$.
\end{proposition*}
Here is the result in the rational case.
\begin{proposition}\label{S}
Let $X$ be a Fano $4$-fold 
and $f\colon X\dasharrow S$ a
quasi-elementary rational contraction onto a surface. Assume that $f$ is not
a morphism. 

If $f$ is not elementary, then $\rho_S\leq 9$ and $\rho_X\leq 17$. 

If $f$ is elementary, then $\rho_X\leq 11$.
\end{proposition}
\begin{proof}
 When $f$ is elementary $\rho_X=\rho_S+1$, while in general
 $\rho_X\leq \rho_S+8$ by Cor.~\ref{pluto}. 
Therefore
we
have to show that $\rho_S\leq 10$ if $f$ is elementary, and
$\rho_S\leq 9$ otherwise.

The surface $S$ is smooth and is a Mori dream space by
Cor.~\ref{target}; moreover $S$
is rational because $X$ is rationally connected.

We  assume that $\rho_X\geq 6$ and $\rho_S\geq 4$. Under these
conditions, we are going to show that $-K_S$ is nef, and ample
when $f$ is not elementary; since $S$ is a smooth rational surface, 
this implies the statement. 
Notice that in order to show that $-K_S$ is nef (respectively, ample), 
it is enough to show
that $-K_S\cdot\sigma\geq 0$ (respectively, $>0$)
for every extremal ray $\sigma$ of
$\NE(S)$; moreover, every such extremal ray corresponds to an
elementary contraction of $S$.

Thus let $g\colon S\to S_1$ be an elementary
contraction.   The surface
 $S_1$ has rational singularities by
Rem.~\ref{ratsing}, and
since $\rho_S\geq 4$, $g$ is birational.

Consider a factorization $X\dasharrow \w{X}\stackrel{\w{f}}{\to}S$ of
$f$ as in Rem.~\ref{ratsing}, and let 
 $C\subset S$ be the irreducible curve contracted by $g$. 
Since $C$ is a non-movable prime divisor in $S$,
by Rem.~\ref{torino}
$D:=(\w{f})^*(C)=(\w{f})^{-1}(C)$ is a non-movable
 prime divisor in $\w{X}$.  
We have $(\w{f})_*(\N(D,\w{X}))=\R[C]\subset\N(S)$, hence
$$
\dim\N(D,\w{X})\leq 1+\dim\ker (\w{f})_*=1+\rho_X-\rho_S\leq\rho_X-3.
$$

Let $D_X\subset X$ be the transform of
$D$.
Since $f$ is not regular, $X$ has a small elementary contraction,
and Cor.~\ref{eco} gives $c_X\leq 2$, hence $\dim\N(D_X,X)\geq\rho_X-2$.
We apply   Th.~\ref{effective} to $D_X$, and consider the possible types.

We notice at once that $\dim\N(D_X,X)>\dim\N(D,\w{X})$, therefore by 
Rem.~\ref{dimension} $D_X$ must contain some exceptional plane. This
implies that
 $D_X$ cannot be of type $(3,2)$ (see Th.~\ref{effective} $(i)$).

We apply Rem.~\ref{factor} to $g\circ \w{f}\colon\w{X}\to S_1$ and $D$, and
get a diagram: 
$$\xymatrix{
{\w{X}}\ar[d]_{\w{f}}\ar@{-->}[r]^h&
{\widehat{X}} \ar[r]^k& {\w{X}_1}\ar[ld]^{\w{f}_1}\\
{S}\ar[r]_{g}&{S_1}& }$$
where $k$ is an elementary
divisorial contraction        
with exceptional divisor the transform of $D$, $\w{f}_1$ is a
contraction, and $h$ is a birational map which factors as a sequence
of $D$-negative flips.
Notice that $\w{X}_1$ is factorial by Cor.~\ref{monaco}, in particular
it is again a Mori dream space (see Rem.~\ref{targetMDS}). 

We show that
$\w{f}_1\colon\w{X}_1\to S_1$ is quasi-elementary.
Let $F\subset\w{X}$ be a general fiber of $\w{f}$, and consider its
transforms $\widehat{F}\subset\widehat{X}$ and $F_1\subset
\w{X}_1$. Since the indeterminacy locus of $h$
 is contained in $D$, it is disjoint from
$F$; therefore $F$, $\widehat{F}$, and $F_1$ are isomorphic, and $F_1$
is a general fiber of $\w{f}_1$.
By Rem.~\ref{divisors} (1) we get
$$\dim\N(\widehat{F},\widehat{X})=\dim\N(F,\w{X})=\rho_X-\rho_S
=\rho_{\w{X}_1}-\rho_{S_1}=\dim\ker(\w{f}_1)_*$$
(we are using that $\w{f}$ is quasi-elementary, see
Prop.~\ref{sabrina} $(ii)$).
  
On the other hand $\widehat{F}\cap\Exc(k)=\emptyset$, therefore
$\N(\widehat{F},\widehat{X})\subseteq\Exc(k)^{\perp}$ and
$\NE(k)\not\subset\N(\widehat{F},\widehat{X})$. We conclude that
$k_*\colon\N(\widehat{X})\to\N(\w{X}_1)$ 
is injective on $\N(\widehat{F},\widehat{X})$, 
and since $\N(F_1,\w{X}_1)=k_*(\N(\widehat{F},\widehat{X}))$, we get
$\dim\N(F_1,\w{X}_1)=\dim\N(\widehat{F},\widehat{X})
=\dim\ker(\w{f}_1)_*$,
and $\w{f}_1$ is quasi-elementary by Prop.~\ref{sabrina} $(ii)$.

\medskip

If $D_X$ is of type $(3,1)$ or $(3,0)^{\pr^3}$, then $\w{X}_1$ is smooth and
it is a SQM of a Fano $4$-fold $X_1$ by Cor.~\ref{monaco}. 
Since $X_1\dasharrow S_1$ is a quasi-elementary rational contraction, 
Cor.~\ref{target} implies that $S_1$ is smooth.
Hence $g$ is the blow-up of a smooth point, and $-K_S\cdot C=1$.

\medskip

Suppose now that $D_X$ is of type $(3,0)^Q$.
We show that $\w{f}_1$ is $K$-negative. 

By contradiction, suppose that
there exists an irreducible curve $C_0\subset\w{X}_1$ such that
$\w{f}_1(C_0)=\{pt\}$ and $-K_{\w{X}_1}\cdot C_0\leq 0$. By
Cor.~\ref{monaco}, $C_0$ cannot contain the singular point $p:=k(\Exc(k))$,
therefore $C_0=k(l)$ where $l\subset\widehat{X}$ is an irreducible
curve, disjoint from $\Exc(k)$, with $-K_{\widehat{X}}\cdot l\leq
0$. By Rem.~\ref{SQM} (2), $l$ is an exceptional line.
We need the following.
\begin{remark}\label{portovenere}
Let $X$ be a Fano $4$-fold and consider a diagram:
\stepcounter{thm}
\begin{equation}\label{diagram}
\xymatrix{
&X\ar@{-->}[dl]_{\ph}\ar@{-->}[dr]^{\psi}&\\
{\w{X}}\ar@{-->}[rr]^{h}&&{\widehat{X}}}\end{equation}
where $\ph$ and $\psi$ are SQMs and $h:=\psi\circ\ph^{-1}$.
 Let $l\subset\widehat{X}$ be an exceptional line. 
\begin{enumerate}[(1)]
\item
Either $l\subset\dom(h^{-1})$, or 
$l\cap\dom(h^{-1})=\emptyset$.
\item 
Let $D$ be a divisor in $\w{X}$, $\widehat{D}$ its transform in
$\widehat{X}$, 
and suppose that $h$ factors as a
sequence of $D$-negative flips. 
If $l\cap\dom(h^{-1})=\emptyset$, then
$\widehat{D}\cdot l>0$. 
\end{enumerate}
\end{remark}
\begin{proof}
By Rem.~\ref{SQM} (2) we have $l\cap\dom(\psi^{-1})=\emptyset$. Therefore if 
$l$ is not contained in the indeterminacy locus of $h^{-1}$, 
then its transform $\w{l}\subset \w{X}$ must be contained in the
indeterminacy locus of  $\ph^{-1}$. Then again by Rem.~\ref{SQM}, $\w{l}$
is an exceptional line, and $h^{-1}=\ph\circ\psi^{-1}$ 
is an isomorphism on $l$.

For the second statement, we can factor $h$ as
$\w{X}\stackrel{h_1}{\dasharrow} \w{X}_1 
\stackrel{h_2}{\dasharrow} \widehat{X}$, where $h_2$ is a $D_1$-negative flip
($D_1$ the transform of $D$ in $\w{X}_1$). By induction, we can assume that
the statement holds for $h_1$. 
Now if $l\cap\dom(h_2^{-1})=\emptyset$, we have $\widehat{D}\cdot l>0$, because
$h_2^{-1}$ is a $\widehat{D}$-positive flip. Otherwise $l$
is contained in the open subset where $h_2^{-1}$ is an isomorphism, so
that $l=h_2(l_1)$, $l_1$ an exceptional line in $\w{X}_1$. Moreover
$l_1\cap\dom(h_1^{-1})=\emptyset$, therefore by 
induction $D_1\cdot l_1>0$ and $\widehat{D}\cdot l>0$. 
\end{proof}
We carry on with the proof of Prop.~\ref{S}, and apply
Rem.~\ref{portovenere} to $h$ and $l\subset\widehat{X}$.
 Since $\Exc(k)\cdot
l=0$, we deduce that $h$ is an isomorphism over
$l$, so that $\w{l}=h^{-1}(l)\subset \w{X}$
is an exceptional line disjoint from $D$ and 
contracted by $g\circ \w{f}$. On the other hand $\dim\w{f}(\w{l})=1$ (because $\w{f}$ is $K$-negative), and $\w{f}(\w{l})\neq C$ (because $\w{l}\not\subset\w{f}^{-1}(C)=D$), thus $\dim (g\circ\w{f})(C)=1$, a contradiction.

\medskip

Hence $\w{f}_1\colon \w{X}_1\to S_1$ is a $K$-negative quasi-elementary contraction. Since $\w{X}_1$ is factorial, as in \cite[Lemmas 3.9 and 3.10]{fanos} one shows that $S_1$ is factorial. Thus
$S_1$ is a normal, Gorenstein surface with rational
singularities, that is, $S_1$ has at most Du Val singularities. 
Therefore either $g\colon S\to S_1$ is the blow-up of a smooth point
and $-K_S\cdot C=1$, or 
$C$ is a $(-2)$-curve in $S$ and $-K_S\cdot C=0$.

\medskip

Summing up, we have shown that  $-K_S\cdot \NE(g)\geq 0$
for every elementary contraction $g$ of $S$, 
 therefore $-K_S$ is nef.

Suppose now that $f$ is not elementary. To show that $-K_S$ is ample,
we need to show that when  $D_X$ is of type $(3,0)^Q$, $C$ cannot
be a $(-2)$-curve. For this, we show the existence of an irreducible
curve $C'\subset\w{X}$ with $D\cdot C'=-1$. This gives:
$$-1=(\w{f})^*(C)\cdot C'=C\cdot (\w{f})_*(C'),$$
hence $(\w{f})_*(C')=C$ and $C^2=-1$.

We know that $\w{f}_1$ is a non-elementary $K$-negative quasi-elementary contraction, so that
the general fiber is a smooth Del Pezzo
surface with Picard number $>1$. In particular, every fiber of
$\w{f}_1$ is covered by curves of anticanonical degree $2$, either
irreducible, or a union of two irreducible curves of anticanonical
degree $1$.

Let $F_0\subset\w{X}_1$ be the fiber containing the singular point $p$.
By Cor.~\ref{monaco}, $p$ cannot be contained in any irreducible
curve of anticanonical degree 
 $2$, hence we find an irreducible curve $C_1\subset F_0$ 
such that $p\in C_1$ and $-K_{\w{X}_1}\cdot
C_1=1$. Again by Cor.~\ref{monaco}, $C_1=k(l_1)$, where
$l_1\subset\widehat{X}$
is an
exceptional line with $\Exc(k)\cdot {l}_1=1$; clearly $l_1\not\subset\Exc(k)$.

Notice that $h$ cannot be an isomorphism over $l_1$, otherwise we
would get an exceptional line in $\w{X}$, not contained in $D$, but
contracted by $g\circ\w{f}$, which is impossible. Therefore by
Rem.~\ref{portovenere} we have $l_1\cap\dom(h^{-1})=\emptyset$.

Consider now the factorization $h=\psi\circ\ph^{-1}$ as in \eqref{diagram},
where $\ph$ and $\psi$ are SQMs.
By Rem.~\ref{SQM} (2), $l_1$ is contained in the indeterminacy locus of
$\psi^{-1}$; let $L\subset X$ be the corresponding exceptional
plane, and $C_L\subset L$ a line. Since $\Exc(k)\cdot l_1=1$ in
$\widehat{X}$, using Rem.~\ref{intersection} we see that  
$D_X\cdot C_L=-1$. Now we cannot
have $L\cap\dom(\ph)=\emptyset$ (otherwise $h$
would be an isomorphism over $l_1$), therefore $L$ intersects the
indeterminacy locus of $\ph$ 
in finitely many points and we can choose $C_L$ disjoint
from it. In the end $C':=\ph(C_L)\subset\w{X}$ is an irreducible
curve with $D\cdot C'=-1$, and  this concludes the proof.
\end{proof}
\subsection{Elementary rational contractions onto $3$-folds}\label{3folds}
In this section we study Fano $4$-folds having an elementary
rational contraction onto a $3$-dimensional variety. 
Also in this case, we first recall the result about the
regular case. 
\begin{thm*}[\cite{fanos}, Cor.~1.2 (iii)]
Let $X$ be a Fano $4$-fold 
and $f\colon X\to Y$ an elementary 
contraction with $\dim Y=3$. Then $\rho_X\leq 11$, with equality only
if $X\cong\pr^1\times\pr^1\times S$ or $X\cong\mathbb{F}_1\times
S$, where $S$ is a surface.
\end{thm*}
Here we show the following.
\begin{thm}\label{dim3}
Let $X$ be a Fano $4$-fold 
and $f\colon X\dasharrow Y$ an elementary rational
contraction with $\dim Y=3$. Then $\rho_X\leq 11$.
\end{thm}
Before proving the theorem, we need some preliminary lemmas.
\begin{lemma}\label{small}
Let $X$ be a Fano $4$-fold
and $X\dasharrow Y$ an elementary rational contraction 
with $\dim Y=3$.
Suppose that $g\colon Y\to Y_0$ is a small elementary contraction.

Then $\Exc(g)$ is the disjoint union of smooth rational curves, lying in
the smooth locus of $Y$, with normal bundle $\mathcal{O}_{\pr^1}(-1)^{\oplus
  2}$; in particular $K_Y\cdot\NE(g)=0$.
\end{lemma}
\begin{proof}
By Rem.~\ref{ratsing}, we can
factor the map $X\dasharrow Y$ as
$X\dasharrow\w{X}\stackrel{f}{\to}Y$, where $\w{X}$ is a SQM of $X$
and $f$ is a $K$-negative
elementary contraction.

By standard properties of $K$-negative elementary
contractions, $f$ is equidimensional except possibly at finitely many
points of $Y$, where $f$ can have isolated $2$-dimensional
fibers. Moreover $Y$ can have at most canonical and factorial
singularities at these points, and is smooth elsewhere (see
Th.~\ref{smallfibers} and Cor.~\ref{target}).  

We have $\dim\Exc(g)=1$ and $g(\Exc(g))=\{p_1,\dotsc,p_r\}$ 
is a finite set of points.
Fix $i\in\{1,\dotsc,r\}$; we show that there exists an exceptional
line $l_i$ contained in  $(g\circ f)^{-1}(p_i)$.

Suppose that this is not the case: then there is an open subset
$U$ of $Y_0$, containing $p_i$, such that
$\widetilde{U}:=(g\circ f)^{-1}(U)$  does not contain exceptional
lines. In particular $(g\circ f)_{|\w{U}}\colon \w{U}\to
U$ is a local contraction and  $-K_{\w{U}}$ is $(g\circ
f)$-ample. Moreover $(g\circ f)_{|\w{U}}$ factors as
$g_{|U_Y}\circ f_{|\w{U}}$, where
$U_Y:=g^{-1}(U)$, so that $\dim\N(\w{U}/U)=2$ (we refer the reader
to \cite{KMM} for the notation in the relative setting).

Let $\tau$ be the extremal ray of
$\NE(\w{U}/U)$ different from $\NE(f_{|\w{U}})$.
We have
$f(\Lo(\tau))\subseteq\Exc(g)$, so that $\dim\Lo(\tau)\leq 2$, and
$\tau$ is a small extremal ray. On the other hand
 $f$ is finite  on the fibers
of the contraction of $\tau$, which then have dimension at most $1$.
Anyway this is impossible by Th.~\ref{smallfibers}, because 
 $-K_{\w{U}}\cdot \tau>0$.

Therefore we have an exceptional line $l_i\subset (g\circ f)^{-1}(p_i)$,
and $g\circ f$ is not $K$-negative. 

\medskip

By flipping the $K$-positive extremal rays contracted by
$g\circ f$ as in the proof of Rem~\ref{ratsing}, 
we get a diagram:
$$\xymatrix{
{\w{X}}\ar[d]_{f}\ar@{-->}[r]^h
&{\widehat{X}}\ar[d]^{\ph}\\   
 {Y}\ar[r]^{g}&{Y_0}
}$$
where $h$ is a composition of $K$-positive flips, and
 $\ph$ is a $K$-negative contraction.
In particular, as in Rem.~\ref{SQM} we see that
$\widehat{X}\smallsetminus\dom(h^{-1})$ is a disjoint union of
exceptional planes, and $\w{X}\smallsetminus\dom(h)$ a disjoint union
of exceptional lines.  

Since $f$ cannot contract any exceptional line,
$h$ is an isomorphism on $(g\circ f)^{-1}(Y_0\smallsetminus
\{p_1,\dotsc,p_r\})$, so that $\ph$ is equidimensional outside a finite subset of
$Y_0$. 

 Fix $i\in\{1,\dotsc,r\}$, set $S_i:=(g\circ f)^{-1}(p_i)$, and let
 $\widehat{S}_i\subset\widehat{X}$ be its transform, so that
 $\widehat{S}_i\subseteq \ph^{-1}(p_i)$. The fiber $\ph^{-1}(p_i)$
 cannot have dimension $3$, because $h$ is an
isomorphism in codimension $1$ and $g\circ f$ has fibers of dimension
at most $2$.
Since $S_i$ has
dimension $2$, $\ph^{-1}(p_i)$ is
an isolated $2$-dimensional fiber of $\ph$. 

On the other hand
 $S_i$ contains the
exceptional line $l_i$, which lies in the indeterminacy locus of $h$.
We conclude that there is an exceptional plane $L_i$, lying in the
indeterminacy locus of $h^{-1}$, and contained in $\ph^{-1}(p_i)$, so that
$\ph^{-1}(p_i)\supseteq L_i\cup \widehat{S}_i$.

We use the classification of possible isolated $2$-dimensional fibers of
$\ph$ given in \cite[Prop.~4.3.1]{AWaview} (notice that
we can apply this result
to $\ph$ using \cite[Th.~2.6]{mella}, as in the proof of
Th.~\ref{components}). 
In particular, if $T$ is an
irreducible component of $\ph^{-1}(p_i)$ which intersects $L_i$ in a
curve, we see that
$T$ is either $\pr^2$, $\pr^1\times\pr^1$, the Hirzebruch surface
$\mathbb{F}_1$, or the quadric cone. On the other hand $T\cap L_i$
must be a negative curve in $T$, therefore the only possibility is
$T\cong\mathbb{F}_1$.  

We conclude from \cite[Prop.~4.3.1]{AWaview} that $\ph^{-1}(p_i)=L_i\cup
\widehat{S}_i$, and either $\widehat{S}_i\cong\pr^2$ intersects $L_i$
in one point, or 
 $\widehat{S}_i\cong\mathbb{F}_1$ intersects  $L_i$ in a curve which
is a line in  
$L_i$, and the $(-1)$-curve in  $\widehat{S}_i$.

In particular $\widehat{S}_i$ is irreducible, therefore $S_i$ is
irreducible and $C_i:=g^{-1}(p_i)$ is an irreducible curve, because
$C_i=f(S_i)$. Moreover $f$ cannot have $2$-dimensional fibers over
$C_i$, because $S_i=f^{-1}(C_i)$, so that $f$ is a conic bundle over
$C_i$ and $C_i\subset Y_{reg}$ 
(see Th.~\ref{smallfibers}). On the other hand $f(l_i)=C_i$ and
$l_i$ cannot intersect curves of anticanonical degree $1$ by
Rem.~\ref{SQM} (1), therefore $f$ is smooth over $C_i$. 

The birational map $h^{-1}$ gives an isomorphism $S_i\smallsetminus
l_i\cong \widehat{S}_i\smallsetminus L_i\cong
\pr^2\smallsetminus\{pt\}$, and under this isomorphism $f_{|S_i\smallsetminus
l_i}$ is the projection. We conclude that  $C_i\cong\pr^1$,
 $S_i\cong\mathbb{F}_1$, and $l_i$ is the
$(-1)$-curve in $\mathbb{F}_1$.

We have $\mathcal{N}_{l_i/S_i}\cong\mathcal{O}_{\pr^1}(-1)$ and 
$\mathcal{N}_{l_i/\w{X}}\cong\mathcal{O}_{\pr^1}(-1)^{\oplus 3}$, which
imply that
$(\mathcal{N}_{S_i/\w{X}})_{|l_i}\cong\mathcal{O}_{\pr^1}(-1)^{\oplus
  2}$. On the other hand
$\mathcal{N}_{S_i/\w{X}}\cong(f_{|S_i})^*\mathcal{N}_{C_i/Y}$, therefore
$$\mathcal{N}_{C_i/Y}\cong
(\mathcal{N}_{S_i/\w{X}})_{|l_i}\cong\mathcal{O}_{\pr^1}
(-1)^{\oplus  2},$$
and this concludes the proof.
\end{proof}
\begin{lemma}\label{divisorial}
Let $X$ be a Fano $4$-fold with $\rho_X\geq 6$ and
$X\dasharrow Y$ an  elementary rational contraction, which is not
regular, 
with $\dim Y=3$.

Suppose that $g\colon Y\to Y_0$ is a divisorial
 elementary contraction.
Then $g$ is the blow-up of a smooth point of $Y_0$; in particular
$-K_Y\cdot\NE(g)>0$.  
\end{lemma}
\begin{proof}
As usual, using Rem.~\ref{ratsing}, we 
factor the map $X\dasharrow Y$ as
$X\dasharrow\w{X}\stackrel{f}{\to}Y$, where $\w{X}$ is a SQM of $X$
and $f$ is a $K$-negative
elementary contraction. Moreover the map $X\dasharrow\w{X}$ is not an
isomorphism.
Since $X$ has a small elementary contraction and $\rho_X\geq 6$, we have
$c_X\leq 2$ by Cor.~\ref{eco}.

By Rem.~\ref{torino},
the divisor
$D:=f^{-1}(\Exc(g))$ is a non-movable prime divisor in $\w{X}$.
Moreover $\N(D,\w{X})=(g_*)^{-1}(\N(\Exc(g),Y))$, so that
\stepcounter{thm}
\begin{equation}\label{venerdi}
\dim\N(D,\w{X})=1+\dim\N(\Exc(g),Y)\leq 3.\end{equation} 
Let $D_X\subset X$ be the transform of $D$; then
$\dim\N(D_X,X)\geq\rho_X-2\geq 4>\dim\N(D,\w{X})$.   
As in the proof of Prop.~\ref{S}, this shows that
$D_X$ cannot be of type $(3,2)$. 

\medskip

\noindent \emph{Step 1: we show that $g$ is of type $(2,0)$.}

\smallskip

\noindent By contradiction, suppose that
$g$ is of type $(2,1)$; we show that then $D_X$ must be
 of type $(3,2)$, which we have already excluded.

Consider the (possibly empty) 
set of exceptional lines
$l_1,\dotsc,l_r\subset\w{X}$ such that $(g\circ f)(l_i)=\{pt\}$. 
 Set $U:=Y_0\smallsetminus 
(g\circ f)(l_1\cup\cdots\cup l_r)$, $U_Y:=g^{-1}(U)$, and
  $\w{U}:=f^{-1}(U_Y)$. Since $Y_0\smallsetminus U$ is a finite set
 and $g$ is of type $(2,1)$, 
we have $\Exc(g)\cap U_Y\neq\emptyset$, and $g_{|U_Y}\colon U_Y\to U$ is a
 non-trivial local contraction. 

 Consider now the local contraction
$(g\circ f)_{\w{U}}\colon \w{U}\to U$.
As in the proof of Lemma~\ref{small}, we see that there is a
birational extremal ray 
$\tau$ of  $\NE(\w{U}/U)$ such that  $-K_{\w{U}}\cdot\tau>0$,
$\tau\neq\NE(f_{|\w{U}})$, and the associated
 contraction has fibers of dimension at most $1$. Then $\tau$ is
 of type $(3,2)$ by Th.~\ref{smallfibers},
in particular $\Lo(\tau)$ is a prime divisor in
 $\w{U}$. On the other hand $f(\Lo(\tau))\subseteq\Exc(g)$,
 therefore $\Lo(\tau)= D\cap\w{U}$. 
 
We run a Mori
program on $\w{X}$ 
for $D$ over $Y_0$. This means that we obtain a commutative diagram:
$$\xymatrix{
{\w{X}=X_0}\ar@{-->}[r]^{f_0}\ar[rrd]_{g\circ f=\ph_0}
&{X_1}\ar@{-->}[r]\ar[rd]^{\ph_1} & 
{\cdots\cdots}\ar@{-->}[r]
&{X_{k-1}}\ar@{-->}[r]^{f_{k-1}}\ar[dl]_{\ph_{k-1}}
&{X_k}\ar[dll]^{\ph_k}\\
&&{Y_0}&&
}$$
 satisfying (\ref{mds}.\theMMPuno) and (\ref{mds}.\theMMPdue),
where moreover for every 
$i=0,\dotsc,k$ there is a contraction $\ph_i\colon X_i\to Y_0$
(with $\ph_0=g\circ f$) such that  $\sigma_i\subseteq\NE(\ph_i)$ for $i<k$. 
Instead of (\ref{mds}.\theMMPtre), in the end we get that either $D_k$
is $\ph_k$-nef, or 
there exists  a $D_k$-negative extremal ray of fiber type 
$\sigma_k\subseteq\NE(\ph_k)$.

In our situation, $D_k$ is effective, therefore in $X_k$
there cannot be a $D_k$-negative extremal ray of fiber
type, and $D_k$ is $\ph_k$-nef.

Let $i\in\{0,\dotsc,k\}$ be such that $f_j$
is a flip  for every $j\in\{0,\dotsc,i-1\}$, and either $i<k$ and $f_i$ is
divisorial, or
$i=k$; in particular  $\w{X}\dasharrow X_i$ is a SQM.
Then  $f_j$
is an isomorphism on $\ph_{j}^{-1}(U)$ for every $j\in\{0,\dotsc,i-1\}$.
Indeed suppose that $i>0$; then $\sigma_0\subset\NE(g\circ f)$ is a
small extremal ray, and $\NE(\w{U}/U)=\tau+\NE(f_{|\w{U}})$, hence
$\Lo(\tau)\cap\w{U}=\emptyset$.  
Iterating this reasoning, in the end we see that
$U_i:=\ph_i^{-1}(U)$ is isomorphic to $\w{U}$, and $D_i\cap U_i$
is the locus of an extremal ray of type $(3,2)$ in $\NE(U_i/U)$.

In particular $D_i$ is not $\ph_i$-nef, so that $i<k$, and
$f_i\colon X_i\to X_{i+1}$  is an
elementary 
divisorial contraction with exceptional divisor $D_i$. We deduce that
that $f_i$ is of type $(3,2)$, and hence that $D_X$ is of type
$(3,2)$, a contradiction.
This concludes the proof of Step~1.

\medskip

Therefore $g$ is of type $(2,0)$; in particular $p:=(g\circ
f)(D)=g(\Exc(g))\in Y_0$
is a point.
We apply Rem.~\ref{factor} to $g\circ f\colon \w{X}\to Y_0$ and $D$, and
 get a commutative diagram:
$$\xymatrix{{\w{X}}\ar@{-->}[r]^h\ar[d]_{f}&{\widehat{X}}\ar[r]^{k}&
{\w{X}_1}\ar[dl]^{f_1}\\
{Y}\ar[r]^{g}&{Y_0}&}$$
where $h$ is a SQM which factors as a
sequence
of $D$-negative flips,
and  $k$ is an elementary divisorial contraction with exceptional
divisor $\widehat{D}$, the transform of $D$. 

Notice that $f_1$ is an elementary $K$-negative contraction of type
$(4,3)$, and that $\w{X}\smallsetminus
D\cong\widehat{X}\smallsetminus(f_1\circ k)^{-1}(p)$.

\medskip

\noindent \emph{Step 2: when $D_X$ is of type $(3,1)$ or
  $(3,0)^{\pr^3}$, then $\w{X}_1$ is smooth and
$\dim f_1^{-1}(p)=1$, so that $Y_0$ is  smooth at $p$.}

\smallskip

\noindent Suppose that $D_X$ is of type $(3,1)$. Then by Cor.~\ref{monaco}
$\w{X}_1$   is smooth
 and $k$ is the blow-up of a
smooth curve $C\subset \w{X}_1$. Moreover $C$ cannot intersect
irreducible
curves
of anticanonical degree $2$, and can intersect only finitely many
irreducible
curves of anticanonical degree $1$. Since the image of $\widehat{D}$
in $Y_0$ is a point, $C$ is contained in a fiber of $f_1$.

Thus $f_1\colon \w{X}_1\to Y_0$ is an elementary $K$-negative contraction
of a smooth $4$-fold, of type $(4,3)$. We know that $f_1$ can have
isolated $2$-dimensional fibers, and that $Y_0$ is smooth outside
their images (see Th.~\ref{smallfibers}). 
Moreover the possible $2$-dimensional fibers have been
classified by Kachi \cite{kachi} and Andreatta and Wi{\'s}niewski
\cite[Prop.~4.3.1]{AWaview}. 
It is not difficult to see that
if $C$ were contained in a $2$-dimensional fiber, 
in any case $C$ should intersect curves of
anticanonical degree $2$, or infinitely many curves of anticanonical
degree $1$, which is impossible.

Hence $C$ is contained in a $1$-dimensional fiber of $f_1$,
 $Y_0$ is smooth
in $p=f_1(C)$, and $g$ is just the blow-up of $p$.

\smallskip

Suppose that $D_X$ is of type $(3,0)^{\pr^3}$. Again by
Cor.~\ref{monaco},
$\w{X}_1$ is smooth and
 $k$ is the blow-up of a point
$q\in \w{X}_1$. Moreover $q$ cannot belong to irreducible
curves of anticanonical
degree $1$, and can belong at most to finitely many irreducible
curves of
anticanonical degree $2$. Similarly to the previous case, using
Th.~\ref{components} on isolated $2$-dimensional fibers of $f_1$,
 we see that $q$
belongs to a $1$-dimensional fiber of $f_1$,
so that
$p=f_1(q)$ is a smooth point of $Y_0$ and $g$ is just a blow-up.

\medskip

\noindent \emph{Step 3: the case where $D_X$ is of type $(3,0)^Q$.}

\smallskip

\noindent For the rest of the proof, we assume 
that $D_X$ is of type $(3,0)^Q$. By
Cor.~\ref{monaco} we know that $\widehat{D}$ is isomorphic to an
irreducible quadric, and 
$q:=k(\widehat{D})\in\w{X}_1$
 is an isolated terminal and factorial singularity.
Moreover we have the following properties.
\begin{enumerate}[(P1)]
\item\label{P1} The point
$q$ cannot belong to irreducible curves of anticanonical degree $2$, and
 can belong at
most to finitely many irreducible
curves of anticanonical degree $1$. 
\item\label{P2}  Let
$C\subset\w{X}_1$ 
be an irreducible curve such that $q\in C$ and $-K_{\w{X}_1}\cdot
C=1$. Then the transform $\widehat{C}\subset\widehat{X}$ is an
exceptional line, and $\widehat{D}\cdot \widehat{C}=1$. 
\item\label{P3} Let $C_1,C_2\subset\w{X}_1$ be distinct
  irreducible curves such that $-K_{\w{X}_1}\cdot
C_1=1$ and the transform 
$\widehat{C}_2\subset\widehat{X}$ of $C_2$ is an
exceptional line.
 Then either $C_1\cap C_2=\emptyset$, or
$C_1\cap C_2=\{q\}$.
\end{enumerate}
Indeed (P\ref{P1}) and (P\ref{P2}) follow directly from
Cor.~\ref{monaco}. For (P\ref{P3}), let
$\widehat{C}_1\subset\widehat{X}$ be the transform of $C_1$. If
$q\not\in C_1$, then $-K_{\widehat{X}}\cdot \widehat{C}_1=1$, so that
$\widehat{C}_1\cap\widehat{C}_2=\emptyset$ by Rem.~\ref{SQM} (1), and 
 $C_1\cap C_2=\emptyset$. If $q\in C_1$, then $\widehat{C}_1$ is an
exceptional line by (P\ref{P2}), therefore
$\widehat{C}_1\cap\widehat{C}_2=\emptyset$ again by Rem.~\ref{SQM},
and $C_1\cap C_2=\{q\}$.

\medskip

\noindent \emph{Step 4: let $T$ be an irreducible component of
  $f_1^{-1}(p)_{red}$ containing $q$. If $\dim T=1$, then $-K_{\w{X}_1}\cdot
  T=1$. If $\dim T=2$, then $T\cong\mathbb{F}_r$ for some $r\geq 0$,
  and the fibers of 
  the $\pr^1$-bundle on $T$ have 
anticanonical degree $1$ in $\w{X}_1$.}

\smallskip

\noindent Since $f_1\colon\w{X}_1\to Y_0$ is an elementary contraction
of type $(4,3)$, it has fibers of dimension at most $2$, and can have
at most isolated $2$-dimensional fibers. Moreover by Th.~\ref{smallfibers}
the general fiber of $f_1$ is a smooth rational curve of
anticanonical degree $2$.
 
By degeneration (for instance using the Hilbert scheme), we find a
connected curve $C\subset \w{X}_1$ containing $q$ and numerically
equivalent to a general fiber of $f_1$, so that $C\subseteq
f_1^{-1}(p)$. Let $C_0$ be an irreducible component of $C$ containing
$q$. We have 
 $-K_{\w{X}_1}\cdot C_0\leq-K_{\w{X}_1}\cdot C= 2$,
$-K_{\w{X}_1}\cdot C_0>0$ because $f_1$ is elementary, and
$-K_{\w{X}_1}\cdot C_0\in\Z$  because $\w{X}_1$ is factorial.  
Using (P\ref{P1}) we conclude that
$-K_{\w{X}_1}\cdot C_0=1$. Thus if $\dim T=1$, we have
$T=C_0$ and we are done.

If $\dim T=2$, the possibilities for $(T,(-K_{\w{X}_1})_{|T})$ are
given by Th.~\ref{components} $(i)$, $(ii)$, or $(iii)$. However 
$(i)$ is
excluded by (P\ref{P1}). In case $(ii)$, 
again
by (P\ref{P1}) $q$ cannot be the vertex of the cone, 
and $q$ cannot be another point of the cone by (P\ref{P2}) and (P\ref{P3})
(just take 
 the line through $q$ and another line). Thus we are left with
 $(iii)$, which gives Step~4.

\medskip

\noindent \emph{Step 5: the contraction
$f_1\circ k\colon\widehat{X}\to Y_0$ is not
  $K$-negative. If $l_1,\dotsc,l_s\subset\widehat{X}$ are the
  exceptional lines contracted by $f_1\circ k$, we have 
  $l_1\equiv\cdots\equiv l_s$,  
 $\widehat{D}\cdot l_j=1$, $-K_{\w{X}_1}\cdot k(l_j)=1$, and $[l_j]$
  belongs to an extremal ray $\sigma$ of $\NE(\widehat{X})$ such that
  $\NE(f_1\circ k)=\NE(k)+\sigma$.}

\smallskip

\noindent We know from Step~4 that $f_1^{-1}(p)$ contains an irreducible
curve of anticanonical degree $1$ through $q$. By (P\ref{P3}), this
gives an exceptional line in $\widehat{X}$ contracted by $f_1\circ k$,
so  $f_1\circ k$ is not $K$-negative. Thus $\NE(f_1\circ
k)=\NE(k)+\sigma$, where $\sigma$ is an extremal ray with
$-K_{\widehat{X}}\cdot\sigma\leq 0$, and by Rem.~\ref{SQM} (2)
$\Lo(\sigma)$ is a disjoint union of numerically equivalent
exceptional lines.

Fix $j\in\{1,\dotsc,s\}$. The image
 $k(l_j)\subset\w{X}_1$ is an
irreducible curve contained in a fiber of $f_1$, so that
$-K_{\w{X}_1}\cdot k(l_j)>0$, while $-K_{\widehat{X}}\cdot l_j=-1$. 
Therefore $l_j\cap\widehat{D}\neq\emptyset$
and $q\in k(l_j)$, in particular $k(l_j)\subseteq f_1^{-1}(p)$. 

By Step~4, if $k(l_j)$ is an irreducible component of $f_1^{-1}(p)_{red}$, then
$-K_{\w{X}_1}\cdot 
  k(l_j)=1$. Otherwise, $k(l_j)$ is contained in a 
$2$-dimensional component $T\cong \mathbb{F}_r$ for some $r\geq 0$.
By (P\ref{P3}) $k(l_j)$ can intersect 
the fibers of the $\pr^1$-bundle on $T$ only in the point $q$.
 Therefore $k(l_j)$ is the fiber of the $\pr^1$-bundle
through $q$ and again
$-K_{\w{X}_1}\cdot k(l_j)=1$.
We deduce that  
$\widehat{D}\cdot l_j=1$ by (P\ref{P2}).

Now notice  that
$\ker(f_1\circ k)_*$ is $2$-dimensional and is generated by $[l_1]$
and $[B]$, where $B$ is a line in the quadric
$\widehat{D}$. We have $\widehat{D}\cdot B=-1$,
$-K_{\widehat{X}}\cdot B=2$, and  $-K_{\widehat{X}}\cdot
l_1=-1$. Thus $[\widehat{D}]$ and $[K_{\widehat{X}}]$ give linearly
independent linear functions on  $\ker(f_1\circ k)_*$, and since 
$l_1,\dotsc,l_s$ have the same intersection with both, we get
$l_1\equiv\cdots\equiv l_s$. Moreover $\sigma$ contains the class of
at least one 
exceptional line, therefore $[l_j]\in\sigma$.

\medskip

\noindent \emph{Step 6: we show that $h$ is just one $D$-negative and
  $K$-negative flip.}

\smallskip

\noindent First of all 
notice that $D\subset\w{X}$ cannot be isomorphic to a quadric
(\emph{e.g.}\ because it has a morphism onto $\Exc(g)$), so that 
$h$ is not an
isomorphism. Let's factor $h$ as
$\w{X}\stackrel{h'}{\dasharrow}\w{X}'\stackrel{h''}{\dasharrow}\widehat{X}$,
where $h'$ is a sequence of $D$-negative flips, and $h''$ is just one 
$D'$-negative flip,  $D'\subset\w{X}'$ the
transform of $D$.
We get a commutative diagram:
$$\xymatrix{
{\w{X}} \ar@/^1pc/@{-->}[rr]^{h}
\ar@{-->}[r]_{h'}\ar[d]_f &{\w{X}'}\ar@{-->}[r]_{h''}\ar[d]^{\ph}&{\widehat{X}}
\ar[d]^k\\
Y\ar[r]^g&{Y_0}&{\w{X}_1}\ar[l]_{f_1}}$$
where  $\ph$ is a
contraction.

Notice that $(h'')^{-1}$ is the flip of a small extremal ray in
$\NE(f_1\circ k)$. By Step~5 $\NE(f_1\circ k)=\NE(k)+\sigma$ and $k$
is a divisorial contraction, therefore  $(h'')^{-1}$ is the flip of
$\sigma$. Since $K_{\widehat{X}}\cdot\sigma>0$, we see that $h''$ is
the flip of a $K$-negative small extremal ray
$\sigma'\subset\NE(\ph)$. Thus we are left 
 to show that $h'$ is
an isomorphism.

\medskip

We show that $\ph$ is $K$-negative. If not,
by Rem.~\ref{SQM} (2) there exists an exceptional line $l'\subset\w{X}'$ such
that $\ph(l')=\{pt\}$. Since $h''$ is a
$K$-negative flip, by Th.~\ref{flip} $\w{X}'\smallsetminus\dom(h'')$ is a
union of exceptional planes, and by Rem.~\ref{SQM} (3) we get $l'\subset\dom(h'')$.
 Therefore the image of $l'$ in $\widehat{X}$ is an
exceptional line contracted by $f_1\circ k$, but whose class is not in
$\sigma$, which contradicts Step~5.

Hence  $\ph$ is $K$-negative, and $\NE(\ph)=\sigma'+\tau$ where $\tau$ is
a $K$-negative extremal ray.

\medskip

Suppose by contradiction
that $h'$ is not an isomorphism. Then $\NE(\ph)$ must contain the
$D'$-positive small extremal ray corresponding to the last flip in the
factorization of $h'$. Since  $\NE(\ph)=\sigma'+\tau$ and
$D'\cdot\sigma<0$, we deduce that $\tau$ is small, $D'\cdot\tau>0$, 
 and $\Lo(\tau)\subset \ph^{-1}(p)$.

In particular,
$\Lo(\tau)$ is a union of exceptional planes which intersect
${D}'$ (see Th.~\ref{flip}). Let $L$ be one of these exceptional planes.

Since also $\Lo(\sigma')$ is a union of exceptional planes, and
$\tau\neq\sigma'$, we have
$\dim(L\cap\Lo(\sigma'))\leq 0$, while $\dim(L\cap D')\geq 1$. Thus 
 the transform 
 $\widehat{L}\subset\widehat{X}$ of $L$ intersects $\widehat{D}$, and
 is 
contained in $(f_1\circ k)^{-1}(p)$. 
Moreover we can find curves in $\widehat{L}$ having positive
intersection with $\widehat{D}=\Exc(k)$, thus
$\widehat{L}\not\subset\widehat{D}$ and $\dim k(\widehat{L})=2$.

Therefore $k(\widehat{L})$ is an irreducible
 component of $f_1^{-1}(p)_{red}$ containing $q$, and by Step~4
we have $k(\widehat{L})\cong\mathbb{F}_r$ for some $r\geq 0$.
Let $C_1,C_2\subset k(\widehat{L})$ two fibers of the $\pr^1$-bundle
not containing $q$. Then their transforms
$\widehat{C}_1,\widehat{C}_2\subset \widehat{X}$ are disjoint and have
anticanonical degree $1$, so they do
not intersect $\Lo(\sigma)$ by Rem.~\ref{SQM} (1). This 
yields 
two disjoint curves in $L\cong\pr^2$, and we have
a contradiction.

\medskip

\noindent \emph{Step 7: $f_1^{-1}(p)$ is a one-dimensional reducible
  fiber of $f_1$, and $s=2$.}

\smallskip

\noindent Since $(g\circ f)^{-1}(p)=D$ and $h$ is just one flip, we have
$(f_1\circ k)^{-1}(p)_{red}=\widehat{D}\cup l_1\cup\cdots\cup l_s$ and
$f_1^{-1}(p)_{red}=k(l_1)\cup\cdots\cup k(l_s)$. We know from Step~5 that
 $-K_{\w{X}_1}\cdot k(l_j)=1$ for every $j=1,\dotsc,s$, while 
 $-K_{\w{X}_1}\cdot f_1^{-1}(p)=2$, so that $s\leq 2$.

Consider now the resolution of the flip $h$ (see Th.~\ref{flip}). We
get a commutative diagram:
$$\xymatrix{&Z\ar[dl]_{\ph}\ar[dr]^{\psi}& &\\
{\w{X}}\ar@{-->}[rr]^h\ar[dr]_{f}&&{\widehat{X}}\ar[r]^{k}& {\w{X}_1}
\ar[dl]^{f_1}\\
&{Y}\ar[r]^{g}&{Y_0}&
}$$
where $\ph$ and $\psi$ are the blow-ups of the indeterminacy loci of
$h$ and $h^{-1}$ respectively. We have 
$$(f_1\circ
k\circ\psi)^{-1}(p)=(g\circ f\circ\ph)^{-1}(p)=\ph^{-1}(D),$$
so that $f_1^{-1}(p)$ cannot be everywhere non-reduced and $s=2$.

\medskip

\noindent \emph{Step 8: we show the statement.}

\smallskip

\noindent We have $\Lo(\sigma)=l_1\cup l_2$
 and $f_1^{-1}(p)=k(l_1)\cup k(l_2)$. 
By the explicit description of the flip $h$
(see Th.~\ref{flip}), and since 
 $\widehat{D}\cdot l_j=1$, we know that
 $D$ is the blow-up of the
(possibly singular but irreducible) quadric $\widehat{D}$
in two smooth points. Let
$L_1,L_2\subset D$ be the exceptional planes; notice that $L_1$ and
$L_2$ lie in the smooth locus of $D$ and are Cartier divisors in $D$.
 Moreover we have $L_1\cup
L_2=\Lo(\sigma')$.

Let $C_{L_i}\subset L_i$ be lines; we have $C_{L_1}\equiv C_{L_2}$ and 
$D\cdot
C_{L_1}=-1$ because $\widehat{D}\cdot l_1=1$ 
(see Rem.~\ref{intersection}). Let moreover
${B}\subset\widehat{D}$ be a general line and $B_0\subset D$
its transform; recall that $-K_{\widehat{X}}\cdot B=2$. Finally
let $F_0\subset\w{X}$ be a
general fiber of $f$, so that $k(h(F_0))$ is a general fiber of $f_1$.
We have:
$$
k\left(h(F_0)\right)\equiv 2k(l_1)\,\text{ in }\,\w{X}_1,\quad
h(F_0)\equiv 2l_1+2B\,\text{ in }\,\widehat{X},\quad\text{and}\quad
F_0\equiv 2B_0-2C_{L_1}\,\text{ in }\,\w{X}.
$$

 Consider now $f_{|D}\colon D\to \Exc(g)$. We have $f(L_i)=\Exc(g)$,
and every fiber of $f_{|D}$ has
dimension one (for instance because a $2$-dimensional fiber should intersect
$L_1$ in a curve, which is impossible). Let $F_D\subset D$ be a fiber
of $f_{|D}$; then $F_D\equiv F_0\equiv 2B_0-2C_{L_1}$.

If $i\colon D\hookrightarrow\w{X}$ is the inclusion and
$i_*\colon\N(D)\to\N(\w{X})$ the associated push-forward of
$1$-cycles, we have $\ker i_*=\R([C_{L_1}]-[C_{L_2}])$ (because
$\dim\N(D)=3$ and $\dim\N(D,\w{X})=2$ by  \eqref{venerdi}).
In particular we get:
$$F_D\equiv_D  2B_0-2C_{L_1}+\lambda(C_{L_1}-C_{L_2}),$$
where $\lambda\in\R$ and $\equiv_D$ denotes numerical equivalence in
$D$.  This gives  
$(L_1\cdot F_D)_D=2-\lambda$ and 
$(L_2\cdot F_D)_D=\lambda$ (where $(\ \cdot\ )_D$ denotes
intersection in $D$),
  so that $\lambda=1$ and  $(L_1\cdot F_D)_D=1$. Therefore $f_{|L_1}\colon
  L_1\to\Exc(g)$ is an isomorphism, $\Exc(g)\cong\pr^2$, and
  $f(C_{L_1})$ is a line in $\Exc(g)$. Moreover
$\Exc(g)\cdot f(C_{L_1})=D\cdot C_{L_1}=-1$, hence
 $g$ is the blow-up of a smooth
point in $Y_0$.
\end{proof}
\begin{proof}[Proof of Th.~\ref{dim3}]
By Cor.~\ref{target}, $Y$ has at most isolated canonical and factorial
singularities, and is a Mori dream space.  
If $f$ is regular, then $\rho_X\leq 11$ by
\cite[Cor.~1.2 (iii)]{fanos}.

Suppose that $Y$ has an elementary rational contraction of fiber type
$g\colon
Y\dasharrow Z$. Then $g\circ f\colon X\dasharrow Z$
is a quasi-elementary rational contraction 
(see Rem.~\ref{pippo} and
Rem.~\ref{composition}), and $\rho_X-\rho_Z=2$.
If $\dim Z\leq 1$, then $\rho_Z\leq 1$ and $\rho_X\leq 3$. 
If instead $\dim Z=2$, 
 Prop.~\ref{S} yields $\rho_Z\leq 9$ and
$\rho_X\leq 11$.

\medskip

Therefore we can assume that 
$f$ is not regular and
 $Y$ has no elementary
rational contraction of fiber type; 
let us also assume that
$\rho_X\geq 6$.

Let $h\colon Y\dasharrow\w{Y}$ be a SQM. Then $h\circ f\colon
X\dasharrow\w{Y}$ is an elementary rational contraction (see
Rem.~\ref{targetMDS}), so that again by Cor.~\ref{target}
 $\w{Y}$ has at most isolated canonical and factorial
singularities. 

We notice that $h\circ f$ cannot be regular. Indeed $f$ is not regular
over some exceptional plane $L\subset X$, such that the lines
contained in $L$ have numerical class in some extremal ray $\sigma$ of
$\NE(X)$. If $h\circ f$ were a morphism, it would be an elementary contraction of fiber type. In particular we would have $\NE(h\circ f)\neq\sigma$,
 so $h\circ f$ should be finite on
$L$, and $\dim(h\circ f)(L)=2$. Thus $h^{-1}\colon\w{Y}\dasharrow Y$
 should be regular on
an open subset of $(h\circ f)(L)$, and $f$ should be regular on an
open subset of $L$, a contradiction.

Consider an elementary contraction $g\colon\w{Y}\to Y_0$. By our
assumptions, $g$ must be birational, 
 therefore Lemmas~\ref{small} and
\ref{divisorial} apply. We deduce that 
either $g$ is the
blow-up of a smooth point of $Y_0$, or $\Exc(g)$ a disjoint
union of smooth rational curves, lying in the smooth locus of $\w{Y}$,
with normal bundle $\mathcal{O}_{\pr^1}(-1)^{\oplus 2}$ -- we call such
a curve a \emph{$(-1,-1)$-curve}. 
We also notice that in the case of the blow-up we have
$K_{\w{Y}}=g^*K_{Y_0}+2\Exc(g)$, hence $[K_{\w{Y}}]\not\in g^*(\N(Y_0))$.

Therefore $-K_{\w{Y}}\cdot\NE(g)\geq 0$ for every elementary contraction
$g$ of $\w{Y}$, and we deduce that $-K_{\w{Y}}$ is nef. In particular we
can take $\w{Y}=Y$, so that
 $-K_Y$ is nef. Moreover,
 since $h\colon Y\dasharrow\w{Y}$ factors as a sequence of flips of small
extremal rays as above, it is not difficult to see that
for every
irreducible curve $C\subset Y$ such that $C\cap\dom(h)\neq\emptyset$,
 we have
$-K_Y\cdot C=-K_{\w{Y}}\cdot\w{C}$, where $\w{C}\subset\w{Y}$ is the
 transform of $C$.

\medskip

By our assumptions, 
there exists a non-movable prime divisor $E\subset Y$ (otherwise
$\Mov(Y)=\Eff(Y)$ and  
Cor.~\ref{tobia} would yield 
an elementary rational contraction of fiber type on $Y$).
Applying
Rem.~\ref{nonmovable}, we find  a SQM $h_0\colon Y\dasharrow\w{Y}_0$ such that
the transform $\w{E}\subset\w{Y}_0$ of $E$ is the exceptional divisor of
an elementary divisorial contraction, so that
$\w{E}\cong\pr^2$ and 
$\mathcal{N}_{\w{E}/\w{Y}_0}\cong\mathcal{O}_{\pr^2}(-1)$. 

\medskip

Consider now  the contraction $\ph\colon Y\to
T$ defined by $\NE(Y)\cap K_Y^{\perp}$.
We show that $\ph$ is birational, \emph{i.e.} that $-K_Y$ is big.  Since $h_0$
factors as a sequence of $K$-trivial flips, the map $\w{\ph}:=\ph\circ
h_0^{-1}\colon \w{Y}_0\to T$ is regular, and $-K_{\w{Y}_0}$ is the
pull-back of some ample Cartier divisor on $T$. 
In particular
$\w{\ph}$ is finite on $\w{E}$, so that
$\dim\w{\ph}(\w{E})=2$. This also shows that $\ph$ is generically
finite on $E$.

By
contradiction, if $\ph$ is of fiber type, then $T=\w{\ph}(\w{E})$ and $\rho_T=1$.
In particular, $\R_{\geq 0}[-K_Y]=\ph^*(\Nef(T))$ is
a one-dimensional cone in $\mathcal{M}_Y$.
On the other hand, since $-K_Y$ is not
big, this cone 
must lie on the boundary of $\Eff(Y)$, and hence on the
boundary of $\Mov(Y)$. Therefore we can choose a cone 
$\tau\in\mathcal{M}_Y$ of dimension $\rho_Y-1$,
containing $\R_{\geq 0}[-K_Y]$, and lying on the boundary of
$\Mov(Y)$. The corresponding rational contraction $g_1\colon
Y\dasharrow 
Y_1$   is elementary, and cannot be small (see Ex.~\ref{example}) nor of
fiber type (by our assumptions), therefore it is divisorial. 
On the other hand if $H\subset \N(Y)$ is the linear span of $\tau$, we have
$[K_Y]\in 
H=g_1^*(\Nu(Y_1))$, and this contradicts our previous description of
elementary divisorial rational contractions of $Y$.

Therefore $-K_Y$ is nef and big, namely $Y$ is an \emph{almost Fano
variety}, and $\ph$ is birational. Moreover $\dim\Exc(\ph)\leq 1$,
because we have already shown that $\ph$ is generically finite on
every non-movable prime divisor. 

\medskip

We are going to proceed similarly to the proof
of \cite[Prop.~2.8]{priska}.
Let $\sigma_1,\dotsc,\sigma_r$ 
be the divisorial extremal rays of $\NE(Y)$, 
and set $E_i:=\Lo(\sigma_i)$.
 Then $E_1,\dotsc,E_r$ are pairwise disjoint, so that
 $E_i\cdot\sigma_j=0$ if $i\neq j$. It is then easy to
 see\footnote{See \emph{e.g.}\ \cite[Rem.~4.6]{31} for a similar statement.}
 that
$\sigma_1+\cdots+\sigma_r$ is an $r$-dimensional face of $\NE(Y)$,
 whose 
contraction $k\colon Y\to Y_r$ is just the blow-up of $r$
distinct smooth points of $Y_r$. 

Notice that $Y_r$ has isolated canonical
and factorial singularities, and is a Mori dream space by
Rem.~\ref{targetMDS}. Since $k^*(-K_{Y_r})=-K_Y+2(E_1+\cdots+E_r)$, we
see that $-K_{Y_r}$ is nef, and that if $C\subset Y_r$ is an
irreducible curve containing some
 point blown-up by $k$, then $-K_{Y_r}\cdot C\geq 2$.
Moreover we have:
$$
\rho_X=\rho_{Y_r}+r+1\quad\text{and}\quad (-K_Y)^3=(-K_{Y_r})^3-8r,
$$
in particular $(-K_{Y_r})^3\geq (-K_Y)^3>0$, so that
$-K_{Y_r}$ is big, and $Y_r$ is again almost Fano.
It is shown in \cite{prok} that
$(-K_{Y_r})^3\leq 72$, which yields
 $r\leq 8$ and $\rho_X\leq\rho_{Y_r}+9$.

\medskip

There exists some extremal ray $\tau$ of $\NE(Y_r)$ with
$-K_{Y_r}\cdot\tau>0$; let  $\pi\colon Y_r\to Z$ be
the corresponding contraction.
We show that $\dim Z\leq 1$, excluding by contradiction all the other
cases. This gives $\rho_{Y_r}\leq 2$ and $\rho_X\leq 11$, 
and concludes the proof.

Suppose first that $\pi$ is birational. If 
$\Exc(\pi)\cap k(\Exc(k))=\emptyset$, we  get a $K$-negative, birational
extremal ray $\sigma'$ of $\NE(Y)$ different from
$\sigma_1,\dotsc,\sigma_r$, a contradiction. Therefore 
 $\Exc(\pi)$ must contain some of the points blown-up by
$k$. 

If $\pi$ is not of type $(2,0)$, then
every non-trivial fiber $F$ of $\pi$ has dimension $1$, and by
\cite[Cor.~1.15]{AWaview} we have $F\cong\pr^1$ and $-K_{Y_r}\cdot
F=1$. In particular, $F$ cannot contain any point blown-up by $k$, so
that $\Exc(\pi)\cap k(\Exc(k))=\emptyset$, a contradiction.

If $\pi$ is of type $(2,0)$, 
the possibilities for $\Exc(\pi)$ and $(-K_{Y_r})_{|\Exc(\pi)}$ are given
by Th.~\ref{components}. We see that the only case where $\Exc(\pi)$ is
not covered by curves of anticanonical degree $1$ is when
$\Exc(\pi)\cong\pr^2$ and $(-K_{Y_r})_{|\Exc(\pi)}=\mathcal{O}_{\pr^2}(2)$. 
On the other hand, in this case the transform of $\Exc(\pi)$ in $Y$
would be covered by curves of anticanonical degree zero, which
contradicts the fact that $\Exc(\ph)$ contains no divisors.

\medskip

Finally, suppose that 
$\dim Z=2$. By Th.~\ref{smallfibers}, the general fiber of $\pi$ is a
smooth rational curve of anticanonical degree $2$, therefore
$-K_{Y_r}\cdot F=2$ for every fiber $F$ of $\pi$.

 For every
$i=1,\dotsc,r$ let $F_i$ be the fiber of $\pi$ through the point
 $k(E_i)$. Since $k(E_i)$ cannot be contained in curves of
 anticanonical degree one, $F_i$ must be an integral fiber;
let $C_i\subset Y$ 
be its transform. The formula $k^*(-K_{Y_r})=-K_Y+2(E_1+\cdots+E_r)$ gives:
$$-K_{Y}\cdot C_i=0, \quad E_i\cdot
C_i=1,\quad\text{and}\quad E_i\cdot C_j=0\ \text{ if }\ i\neq j;$$ in particular
  $[C_1],\dotsc,[C_r]$ are
 linearly independent in $\N(Y)$.

Consider now the contraction $\pi\circ k\colon Y\to Z$, and
the face $\eta:=\NE(\pi\circ k)\cap
K_Y^{\perp}$ of $\NE(Y)$. The unique
irreducible curves of anticanonical degree zero contracted by
$\pi\circ k$ are $C_1,\dotsc,C_r$, therefore 
$\eta=\R_{\geq 0}[C_1]+\cdots+\R_{\geq 0}[C_r]$ is an
$r$-dimensional face of $\NE(Y)$. This implies that each $\R_{\geq 0}[C_i]$ is
an extremal ray of $\NE(Y)$, and $C_i$ is a $(-1,-1)$-curve.

We claim that there exists a SQM $Y\dasharrow\widehat{Y}$ whose
indeterminacy locus is exactly $C_1\cup\cdots\cup C_r$; this can be
constructed inductively as follows. 

Take a nef divisor  $H$ in $Y$
such that $\NE(Y)\cap H^{\perp}=\eta$, 
 consider the flip
$Y\dasharrow Y_1$ of $\R_{\geq 0}[C_1]$,  
 and let
$C_1'\subset Y_1$ be the new $(-1,-1)$-curve. Then
 $[C_1'],[C_2],\dotsc,[C_r]$\footnote{We still denote by $C_i$ the
   transform of $C_i$, for $i=2,\dotsc,r$.} 
are linearly independent in $\N(Y_1)$, and
$H$ yields a nef
 divisor $H_1$ on $Y_1$ such that $\NE(Y_1)\cap H_1^{\perp}=
\R_{\geq 0}[C'_1]
+\R_{\geq 0}[C_2]+\cdots+\R_{\geq 0}[C_r]$. Hence for $i=2,\dotsc,r$ each
$\R_{\geq 0}[C_i]$ stays a small extremal ray in $Y_1$. Now we can
flip $\R_{\geq 0}[C_2]$, and proceed in the same way.

In the end we get a commutative diagram:
$$\xymatrix{Y\ar@{-->}[r]\ar[d]_k & {\widehat{Y}}\ar[d]^{\widehat{k}} \\
{Y_r}\ar[r]^{\pi}& Z}$$
where $\widehat{k}\colon \widehat{Y}\to Z$ is a contraction.

The transform $\widehat{E}_i\subset\widehat{Y}$ of $E_i$
is isomorphic to $\mathbb{F}_1$, and contains a $(-1,-1)$-curve
$\widehat{C}_i$ as the $(-1)$-curve. If $G_i\subset \widehat{E}_i$ is a fiber of
the $\pr^1$-bundle, and $G_0\subset \widehat{Y}$ a general fiber of
$\widehat{k}$, 
it is not difficult to see that $G_0\equiv G_i$, so that
$$\NE(\widehat{k})=\R_{\geq 0}[G_0]+\R_{\geq 0}[\widehat{C}_1]+\cdots+\R_{\geq
  0}[\widehat{C}_r].$$
Since $\dim\NE(\widehat{k})=r+1$, this implies that $\R_{\geq 0}[G_0]$ is an
extremal ray of $\NE(\widehat{Y})$, whose contraction is of fiber type.
Thus $Y$ has an elementary rational contraction of fiber type, which
contradicts our assumptions, and this concludes the proof.
\end{proof}
\begin{proof}[Proof of Th.~\ref{main}]
The statement follows from \cite{fanos} when $X$ has a regular
elementary contraction of fiber type (see the Introduction). The
general statement follows from
Cor.~\ref{pluto}, Prop.~\ref{S}, and Th.~\ref{dim3}. 
\end{proof}
\section{Fano $4$-folds with $c_X=1$ or $c_X=2$}\label{ultima}
In this section we show the following results, which imply Th.~\ref{terzo}.
\begin{proposition}\label{due}
Let $X$ be a Fano $4$-fold with $\rho_X\geq 6$ and $c_X=2$.
Then one of the following holds:
\begin{enumerate}[$(i)$]
\item $\rho_X\leq 12$, and there is a diagram
$$X\longrightarrow X_1\stackrel{h}{\dasharrow}\w{X}_1\longrightarrow Y$$
where $X_1$ is a Fano $4$-fold, $h$ is a SQM,
$\w{X}_1\to Y$ is an elementary contraction and a 
conic bundle, and $X\to X_1$ is the
blow-up of a smooth irreducible surface contained in $\dom(h)$;
\item there exists a Fano $4$-fold $Y$ and $X\to Y$ a blow-up of
  two disjoint smooth irreducible surfaces.
\end{enumerate}
\end{proposition}
\begin{proposition}\label{uno}
Let $X$ be a Fano $4$-fold with $\rho_X\geq 6$ and $c_X=1$. Then one
of the following holds: 
\begin{enumerate}[$(i)$]
\item $\rho_X\leq 11$ and $X$ has a SQM $\w{X}$ with an elementary
  contraction of fiber type $\w{X}\to Y$ which is a conic bundle; 
\item 
$X$ is obtained by blowing-up a Fano $4$-fold $Y$ in a smooth
irreducible  surface.
\end{enumerate}
\end{proposition}
For the proofs of Prop.~\ref{due} and \ref{uno}, we need the following property.
\begin{remark}\label{allafine}
 Let $X$ be a Fano $4$-fold with $c_X\leq 2$ and $\rho_X\geq 6$. 
Let $E\subset X$ be a prime divisor which is a smooth $\pr^1$-bundle, with fiber $F\subset E$, such that $E\cdot F=-1$.
Then $\R_{\geq 0}[F]$ is an extremal ray of type $(3,2)$,
and it is 
the unique $E$-negative extremal ray of $\NE(X)$.
\end{remark}
\begin{proof}
Let $\sigma_1,\dotsc,\sigma_h$ be the $E$-negative extremal rays of $\NE(X)$
 (notice that $h\geq 1$, because $E$ is not nef). Fix
$i\in\{1,\dotsc,h\}$. We have $\Lo(\sigma_i)\subseteq E$.

If $\sigma_i$ is of type $(3,0)$ or $(3,1)$, then $\dim\N(E,X)\leq 2$, a contradiction because $c_X\leq 2$ and $\rho_X\geq 6$.
If $\sigma_i$ is small, then $\Lo(\sigma_i)$ is a union of exceptional planes (by \cite{kawsmall}), which  must intersect every fiber of the $\pr^1$-bundle structure on $E$.  This yields $\dim\N(E,X)=2$, 
again a contradiction.

Therefore $\sigma_i$ is of type $(3,2)$,
$E=\Lo(\sigma_i)$, and $(-K_X+E)\cdot \sigma_i=0$.
This shows that $-K_X+E$ is nef, and 
$\tau:=\sigma_1+\cdots+\sigma_h=(-K_X+E)^{\perp}\cap\NE(X)$ is a face containing 
$[F]$. 

If $\dim \tau>1$, any $2$-dimensional face of
$\tau$ yields a contraction of $X$ onto $Z$ with $\rho_X-\rho_Z=2$, 
sending $E$ to a point or to a
curve.  This implies that $\dim\N(E,X)\leq 3$, again a contradiction. Thus $h=1$ and $\sigma_1=\R_{\geq 0}[F]$. 
\end{proof}
\begin{proof}[Proof of Prop.~\ref{due}]
Let $D\subset X$ be a prime divisor with $\codim\N(D,X)=2$;
we apply \cite[Prop.~2.5]{codim} to $D$.

Suppose first that we
get two disjoint prime divisors $E_1,E_2$ which are smooth $\pr^1$-bundles, with
fibers $F_i\subset E_i$, such that $E_i\cdot F_i=-1$, $D\cdot F_i>0$, and $[F_i]\not\in\N(D,X)$, for $i=1,2$
(that is, $s=2$ in \cite[Prop.~2.5]{codim}).

Fix $i\in\{1,2\}$.
By Rem.~\ref{allafine}, $\R_{\geq 0}[F_i]$ is an extremal ray of type $(3,2)$, and it is the unique $E_i$-negative extremal ray of $\NE(X)$. 
If $F_0$ is a fiber of the associated contraction, then $F_0\cap D\neq \emptyset$ (for $D\cdot F_i>0$), and $\dim F_0\cap D=0$ (for $[F_i]\not\in\N(D,X)$). Therefore $\dim F_0=1$, and the ray  $\R_{\geq 0}[F_i]$ is of type $(3,2)^{sm}$.

This also shows that $-K_X+E_1+E_2$ is nef, and
$(-K_X+E_1+E_2)^{\perp}\cap\NE(X)=\R_{\geq 0}[F_1]+\R_{\geq 0}[F_2]$ is a face of $\NE(X)$. The associated contraction
  $\ph\colon X\to Y$ 
is the smooth blow-up of two disjoint
 irreducible surfaces.
 Moreover $Y$ is Fano, because
$\ph^*(-K_{Y})=-K_X+E_1+E_2$, therefore we have $(ii)$.

\medskip

Suppose now that \cite[Prop.~2.5]{codim} applied to $D$
gives just one prime divisor $E_1$. 
As in the previous case, we see that $E_1$ is the exceptional divisor
of the blow-up $f_0\colon X\to X_1$ of a Fano $4$-fold $X_1$ along a smooth
surface $S=f_0(E_1)$. Moreover we
 are in the situation of
 \cite[Lemma~2.8]{codim}, and we have a sequence:
$$X=X_0\stackrel{f_0}{\longrightarrow}X_1
\stackrel{f_{1}}{\dasharrow}X_2
\dasharrow\cdots\dasharrow
X_{k-1}\stackrel{f_{k-1}}{\dasharrow}X_k\stackrel{f_k}{\longrightarrow}
Y$$
which is a Mori program for $-D$, where $f_k$
 is an elementary contraction of type $(4,3)$, finite on $D_k\subset X_k$.
Finally
$S\subset X_1$ is contained in the open subset where
the birational map $X_1\dasharrow X_k$ is an isomorphism.

If $f_1,\dotsc,f_{k-1}$ are all
flips, then $X_1\dasharrow X_k$ is a SQM, and we get $\rho_{X_1}\leq 11$ by
Th.~\ref{dim3}. Hence $\rho_X\leq 12$ and we have $(i)$.

\medskip
 
Suppose now that $f_1,\dotsc,f_{k-1}$ are not all flips.
Since the map $X_1\dasharrow X_k$ is an isomorphism on $S$, we can
replace the sequence above by:
$$X=X_0\stackrel{g_0}{\dasharrow}X'_1
\stackrel{g_{1}}{\dasharrow}X'_2
\dasharrow\cdots\dasharrow
X'_{k-1}\stackrel{g_{k-1}}{\longrightarrow}X_k\stackrel{f_k}{\longrightarrow}
Y,$$
where $g_{k-1}\colon X'_{k-1}\to X_k$ is the blow-up of the image
of $S$, and $g_0,\dotsc,g_{k-2}$ are not all flips. Notice that the
birational map $X\dasharrow X'_{k-1}$ gives an isomorphism between
$E_1$ and $\Exc(g_{k-1})$.

Let 
 $i\in\{0,\dotsc,k-2\}$ be the first index such that $g_{i}$
is a divisorial contraction.  
We have:
$$X\stackrel{\ph}{\dasharrow} X'_i\stackrel{g_i}{\longrightarrow} 
X'_{i+1}{\dasharrow}  X_k
\stackrel{f_k}{\longrightarrow}Y,$$
where $\ph$ is a SQM.
Since $\rho_{X}\geq 6$, Cor.~\ref{monaco} applies to $g_i$.

Let $E_2\subset X$ be the transform of
$\Exc(g_{i})$, and  $p\in E_2$ a point which does not
belong to any exceptional plane. Notice that $E_1\cap E_2=\emptyset$.

Proceeding as in \cite[proof of Lemma~2.8]{codim}, we construct a curve $C\subset X$
 with the following properties:
\begin{enumerate}[$(1)$]
\item $p\in C$ and $C$ is 
numerically equivalent to a general fiber $C_0$ of the map
$X\dasharrow Y$, so that $-K_X\cdot C=2$ and $E_2\cdot C=0$;
\item $C=C'\cup\w{F}$, where $\w{F}$ is the transform of
 an integral fiber $F\subset X_k$ of $f_k$, $E_2\cdot
\w{F}>0$, and $\w{F}\not\subset E_2$.
\end{enumerate}

Let $\w{F}_i\subset X'_{i}$ and $\w{F}_{i+1}\subset X'_{i+1}$ be the
transforms of $F$. 
 We have $-K_{X'_{i}}\cdot\w{F}_{i}
\leq -K_{X'_{i+1}}\cdot\w{F}_{i+1}
\leq-K_{X_k}\cdot
F=2$ by \cite[Lemma~3.8]{31}, while $-K_{X}\cdot\w{F}=1$, therefore by
Lemma~\ref{SQM} (1) we have two possibilities: 
\begin{enumerate}[$(a)$]
\item  $-K_{X'_{i}}\cdot\w{F}_{i}=2$ and $\w{F}$ intersects a unique
  exceptional plane
$L\subseteq X\smallsetminus\dom(\ph)$;
\item 
$-K_{X'_{i}}\cdot\w{F}_{i}=1$ and $\ph$ is an isomorphism on
$\w{F}$.
\end{enumerate}

We assume first that we are in case $(a)$, and show that this gives a
contradiction. 
Since  $-K_{X'_i}\cdot \w{F}_i=2=-K_{X_k}\cdot
F$, by \cite[Lemma~3.8]{31} the birational map
$X'_{i}\dasharrow X_k$ is an isomorphism on $\w{F}_{i}$; recall that
the image of  $\w{F}_{i}$ in $X_k$ is an integral fiber $F$ of $f_k$. Thus
$\w{F}_i\cap\Exc(g_i)=\emptyset$
  and
$\w{F}_i$ is a proper fiber of the map $X'_i\dasharrow Y$.

On the other hand
 $\w{F}\cap E_2\neq\emptyset$ by (2), therefore $\w{F}$ intersects
$E_2$ along the indeterminacy locus of $\ph$, and we get
$\w{F}\cap E_2\subset L$.

In $X_i'$ we have $\ph(C_0)\equiv \w{F}_i$ (recall that $C_0$ is a
general fiber of $X\dasharrow 
Y$), hence $C_0\equiv \w{F}+C_L$, where 
$C_L\subset L$ is a line. But we also have
$C_0\equiv C=\w{F}+C'$,
so that $C'\equiv C_L$. 

This implies that $C'$ is contained in an exceptional plane too. 
Indeed by taking a
general very ample divisor $H\subset X'_i$, its transform $\w{H}\subset
X$ is a movable divisor whose base locus is $X\smallsetminus\dom(\ph)$, 
and $\w{H}\cdot C'=\w{H}\cdot
 C_L<0$. 

On the other hand we have $p\in C\cap E_2=C'\cup (\w{F}\cap E_2)$, so
$p$ must belong to some exceptional plane, which contradicts our
choice of $p$.

\medskip

Hence we are in case $(b)$. Using (2) we see that 
$\w{F}_{i}\cdot \Exc(g_{i})=\w{F}\cdot E_2>0$, and $\w{F}_{i}$ is not
contained in $\Exc(g_{i})$. Therefore:
$$\w{F}_{i+1}\cap g_i(\Exc(g_i))\neq\emptyset,\quad \w{F}_{i+1}\not\subseteq
g_i(\Exc(g_i)),\quad\text{and}\quad
-K_{X'_{i+1}}\cdot\w{F}_{i+1}\leq 2.$$ 
Then Cor.~\ref{monaco} yields  
that $g_i$ must be of type $(3,2)$, $\ph$
gives an isomorphism between $E_2$ and $\Exc(g_i)$, and $E_2$ does not
contain any exceptional plane.

This implies that
 $-K_{X'_{i+1}}\cdot\w{F}_{i+1}=2=-K_{X_k}\cdot
F$, and again by \cite[Lemma~3.8]{31} the birational map
$X'_{i+1}\dasharrow X_k$ is an isomorphism between $\w{F}_{i+1}$ and
$F\subset X_k$, so that
$\w{F}_{i+1}$ is a fiber of the map $X'_{i+1}\dasharrow Y$.

\medskip

Since $E_2$ does not contain exceptional planes, the choice of $p\in
E_2$ was arbitrary. We
deduce that
$g_i(\Exc(g_i))$ is contained in the
open subset where the map $X'_{i+1}\dasharrow Y$ is regular and proper. 

Finally $g_i$ cannot have fibers of dimension $2$, otherwise the
rational map $X'_i\dasharrow Y$ over an open subset yields a $K$-negative
local contraction of a smooth variety having a $2$-dimensional fiber with a
one-dimensional component, which is impossible, see \cite[Lemma~2.12]{AWaview}.

Therefore $g_i$ is of type $(3,2)^{sm}$, and $E_2$ is a smooth
$\pr^1$-bundle with fiber $F_2\subset E_2$ such that $E_2\cdot F_2=-1$
and $E_1\cap E_2=\emptyset$.
Now proceeding as in the first part of the proof
we show that we are in $(ii)$.
 \end{proof}
The proof of
Prop.~\ref{uno} is very similar to
that of Prop.~\ref{due}. 

\footnotesize
\addcontentsline{toc}{section}{References}
\providecommand{\bysame}{\leavevmode\hbox to3em{\hrulefill}\thinspace}
\providecommand{\MR}{\relax\ifhmode\unskip\space\fi MR }
\providecommand{\MRhref}[2]{%
  \href{http://www.ams.org/mathscinet-getitem?mr=#1}{#2}
}
\providecommand{\href}[2]{#2}

\bigskip

\noindent C.\ Casagrande\\
Dipartimento di Matematica, Universit\`a di Pavia \\
via Ferrata, 1 \\
 27100 Pavia - Italy 

\smallskip

\noindent\emph{Current address:}
\\ Dipartimento di Matematica, Universit\`a di Torino \\
via Carlo Alberto, 10 \\
 10123 Torino - Italy \\
cinzia.casagrande@unito.it

\begin{thebibliography}{BCHM10}

\bibitem[ADHL10]{coxbook}
I.~Arzhantsev, U.~Derenthal, J.~Hausen, and A.~Laface, \emph{Cox rings},
  preprint arxiv:1003.4229v2, 2010.

\bibitem[Ara10]{carolina}
C.~Araujo, \emph{The cone of pseudo-effective divisors of log varieties after
  {B}atyrev}, Math.\ Z.\ \textbf{264} (2010), 179--193.

\bibitem[AW97]{AWaview}
M.~Andreatta and J.~A. Wi{\'s}niewski, \emph{A view on contractions of higher
  dimensional varieties}, Algebraic Geometry - Santa Cruz 1995,
  Proc.~Symp.~Pure Math., vol.~62, 1997, pp.~153--183.

\bibitem[Bar10]{sammy}
S.~Barkowski, \emph{The cone of moving curves of a smooth {F}ano three- or
  fourfold}, Manuscripta Math.\ \textbf{111} (2010), 305--322.

\bibitem[Bat99]{bat2}
V.~V. Batyrev, \emph{On the classification of toric {F}ano 4-folds}, J.\ Math.\
  Sci.\ (New York) \textbf{94} (1999), 1021--1050.

\bibitem[BCHM10]{BCHM}
C.~Birkar, P.~Cascini, C.~D. Hacon, and
  J.~M{\parbox[b][\Mheight][t]{\cwidth}{c}}Kernan, \emph{Existence of minimal
  models for varieties of log general type}, J.\ Amer.\ Math.\ Soc.\
  \textbf{23} (2010), 405--468.

\bibitem[BDPP04]{BDPP}
S.~Boucksom, J.-P. Demailly, M.~Paun, and T.~Peternell, \emph{The
  pseudo-effective cone of a compact {K}{\"a}hler manifold and varieties of
  negative {K}odaira dimension}, preprint arxiv:math.AG/0405285, 2004.

\bibitem[Bel86]{beltrametti86}
M.~C. Beltrametti, \emph{Contractions of non numerically effective extremal
  rays in dimension $4$}, Proceedings of the Conference in Algebraic Geometry
  (Berlin, 1985), Teubner-Texte Math., vol.~92, 1986, pp.~24--37.

\bibitem[Bel87]{beltra}
\bysame, \emph{On d-folds whose canonical bundle is not numerically effective,
  according to {M}ori and {K}awamata}, Ann.\ Mat.\ Pura Appl.\ (4) \textbf{147}
  (1987), 151--172.

\bibitem[Cas08]{fanos}
C.~Casagrande, \emph{Quasi-elementary contractions of {F}ano manifolds},
  Compos.\ Math.\ \textbf{144} (2008), 1429--1460.

\bibitem[Cas09]{31}
\bysame, \emph{On {F}ano manifolds with a birational contraction sending a
  divisor to a curve}, Michigan Math.\ J.\ \textbf{58} (2009), 783--805.

\bibitem[Cas11]{codim}
\bysame, \emph{On the {P}icard number of divisors in {F}ano manifolds},
  preprint arxiv:0905.3239v4, 2011, to appear in
  Ann.~Sci.~{\'E}c.~Norm.~Sup{\'e}r.

\bibitem[CJR08]{priska}
C.~Casagrande, P.~Jahnke, and I.~Radloff, \emph{On the {P}icard number of
  almost {F}ano threefolds with pseudo-index {$>1$}}, Internat.\ J.\ Math.\
  \textbf{19} (2008), 173--191.

\bibitem[Con02]{conrads}
H.~Conrads, \emph{Weighted projective spaces and reflexive simplices},
  Manuscripta Math.\ \textbf{107} (2002), 215--227.

\bibitem[Fuj86]{Fujita86}
T.~Fujita, \emph{Projective varieties of {$\Delta$}-genus one}, Algebraic and
  Topological Theories (Kinosaki, 1984), 1986, pp.~149--175.

\bibitem[Fuj90]{fuji2}
\bysame, \emph{On singular {D}el {P}ezzo varieties}, Algebraic Geometry
  (L'Aquila, 1988), Lecture Notes in Math., vol. 1507, Springer-Verlag, 1990,
  pp.~117--128.

\bibitem[HK00]{hukeel}
Y.~Hu and S.~Keel, \emph{Mori dream spaces and {GIT}}, Michigan Math.\ J.\
  \textbf{48} (2000), 331--348.

\bibitem[IP99]{fanoEMS}
V.~A. Iskovskikh and Yu.~G. Prokhorov, \emph{Algebraic geometry {V} - {F}ano
  varieties}, Encyclopaedia Math.\ Sci.\, vol.~47, Springer-Verlag, 1999.

\bibitem[Kac97]{kachi}
Y.~Kachi, \emph{Extremal contractions from 4-dimensional manifolds to 3-folds},
  Ann.\ Scuola Norm.\ Sup.\ Pisa Cl.\ Sci.\ (4) \textbf{24} (1997), 63--131.

\bibitem[Kaw89]{kawsmall}
Y.~Kawamata, \emph{Small contractions of four dimensional algebraic manifolds},
  Math.\ Ann.\ \textbf{284} (1989), 595--600.

\bibitem[KM98]{kollarmori}
J.~Koll{\'a}r and S.~Mori, \emph{Birational geometry of algebraic varieties},
  Cambridge Tracts in Mathematics, vol. 134, Cambridge University Press, 1998.

\bibitem[KMM87]{KMM}
Y.~Kawamata, K.~Matsuda, and K.~Matsuki, \emph{Introduction to the minimal
  model problem}, Algebraic Geometry, Sendai, 1985, Adv.\ Stud.\ Pure Math.\,
  vol.~10, 1987, pp.~283--360.

\bibitem[Kol86]{kollarhigher}
J.~Koll{\'a}r, \emph{Higher direct images of dualizing sheaves {I}}, Ann.\ of
  Math.\ \textbf{123} (1986), 11--42.

\bibitem[Mel99]{mella}
M.~Mella, \emph{On del {P}ezzo fibrations}, Ann.\ Scuola Norm.\ Sup.\ Pisa Cl.\
  Sci.\ (4) \textbf{28} (1999), 615--639.

\bibitem[Nam97]{namikawa}
Y.~Namikawa, \emph{Smoothing {F}ano 3-folds}, J.\ Algebraic Geom.\ \textbf{6}
  (1997), 307--324.

\bibitem[Pro05]{prok}
Yu.~G. Prokhorov, \emph{On the degree of {F}ano threefolds with canonical
  {G}orenstein singularities}, Mat.\ Sb.\ \textbf{196} (2005), 81--122,
  Russian. English translation: Sb.\ Math.\ {\bf 196} (2005), 77--114.

\bibitem[Tak99]{takagi}
H.~Takagi, \emph{Classification of extremal contractions from smooth fourfolds
  of (3,1)-type}, Proc.\ Amer.\ Math.\ Soc.\ \textbf{127} (1999), 315--321.

\end{thebibliography}
\end{document}